\documentclass[preprint, 12pt]{article}
\label{key}
\usepackage{calc}
\usepackage[all]{xy}
\usepackage[centertags]{amsmath}
\usepackage{latexsym}
\usepackage{amsfonts}
\usepackage{amssymb}
\usepackage{amsthm}
\usepackage{geometry}
\usepackage{fancyhdr}
\usepackage[dvips]{epsfig}
\usepackage{newlfont}
\usepackage{fancyhdr}
\usepackage{color}
\usepackage{blindtext}
\usepackage{graphics}
\usepackage{newlfont}
\usepackage{wrapfig}
\usepackage{floatflt}
\usepackage{floatflt}
\usepackage{pifont}
\usepackage{hyperref}
\usepackage[english]{babel}
\usepackage[latin1]{inputenc}
\usepackage{graphicx}
\usepackage{t1enc}
\usepackage{pifont}
\usepackage{makeidx}
\usepackage{pgf,tikz,pgfplots}
\pgfplotsset{compat=1.13}
\usepackage{mathrsfs}
\usetikzlibrary{arrows}
\usepackage{tikz-cd}
\usepackage{enumitem}
\usepackage{kantlipsum}
\usepackage{authblk}
\usetikzlibrary{arrows}
\fancypagestyle{plain}{ }

\newlength{\defbaselineskip}
\newtheorem{definition}{Definition}
\newtheorem{proposition}[definition]{Proposition}
\newtheorem{remark}[definition]{Remark}
\newtheorem*{Key words}{Key words}
\newtheorem{notation}{Notation}

\newtheorem*{Proof of Theorem Equivalence between realization functors}{Proof of Theorem \ref{Equivalence_between_realization_functors}}
\newtheorem{Proof of proposition Lif adjunction symmetric sequences}{Proof of proposition \ref{lift_adjonction_symmetric_sequence}}
\newtheorem*{Proof of Theorem Cofree coalgebras}{Proof of Theorem \ref{Cofree Cocommutative_Coalgebra}}
\newtheorem*{Proof of main theorem} {Proof of theorem \ref{Quillen_equivalence_operads}} 
\newtheorem{theorem}[definition]{Theorem}
\newtheorem*{Proof of Proposition CoFree Cosimplicial  Cocommutative Coalgebra} {Proof of Proposition \ref{coFree_Cocomutative_Coalgebra}}

\newtheorem{mylemma}[definition]{Lemma}
\newtheorem*{Abstract}{Abstract}
\newtheorem*{Theorem A}{Theorem A}
\newtheorem*{Theorem B}{Theorem B}
\newtheorem*{Proposition A-1}{Proposition A-1}
\newtheorem*{Acknowledgements}{Acknowledgements}
\newtheorem*{Corollary B-1}{Corollary B-1}

\newsavebox{\bparcould}\newlength{\lparcould}

 \theoremstyle{plain}  \makeindex
\makeatletter
\newcommand{\thechapterwords}
{ \ifcase \thechapter\or Premier\or Deux\or Trois\or Quatre\or
	Cinq\or
	Six\or Sept\or Huit\or Neuf\or Dix\or Onze\fi}
\def\thickhrulefill{\leavevmode \leaders \hrule height 1ex \hfill \kern \z@}
\def\@makechapterhead#1{    \vspace*{15\p@}  {\parindent \z@ \centering \reset@font
		\thickhrulefill\quad
		\scshape \@chapapp{} \thechapterwords
		\quad \thickhrulefill
		\par\nobreak
		\vspace*{15\p@}        \interlinepenalty\@M
		\hrule
		\vspace*{15\p@}        \large \bfseries #1\par\nobreak
		\par
		\vspace*{15\p@}        \hrule
		\vskip 60\p@
	}}
	\def\@makeschapterhead#1{    \vspace*{15\p@}  {\parindent \z@ \centering \reset@font
			\thickhrulefill
			\par\nobreak
			\vspace*{15\p@}        \interlinepenalty\@M
			\hrule
			\vspace*{15\p@}        \large \bfseries #1\par\nobreak
			\par
			\vspace*{15\p@}        \hrule
			\vskip 60\p@
		}}
		\def\@makechapterhead#1{    \vspace*{15\p@}  {\parindent \z@ \centering \reset@font
				\thickhrulefill\quad
				\scshape \@chapapp{} \thechapterwords
				\quad \thickhrulefill
				\par\nobreak
				\vspace*{15\p@}        \interlinepenalty\@M
				\hrule
				\vspace*{15\p@}        \large \bfseries #1\par\nobreak
				\par
				\vspace*{15\p@}        \hrule
				\vskip 60\p@
			}}

\begin{document}

				\bigskip
				
				\begin{center}
					
					\textbf{\Large{Equivalences of operads over symmetric monoidal categories}}\\
					
					\bigskip
					
					\title{Equivalences of operads over symmetric monoidal categories}
					\bigskip
					MIRADAIN ATONTSA NGUEMO 
					\bigskip
					
					Universit\'{e} Catholique de Louvain, IRMP, Louvain La Neuve, 1348, Belgium

					\bigskip
					Email: miradain.atontsa@uclouvain.be
					\bigskip
					%	\author[2]{BENOIT FRESSE }
				
				%!TeX spellcheck = en_US	
					%\today
					\bigskip
					\bigskip
				\end{center}
				\begin{Abstract}In this paper, we study conditions for extending Quillen model category properties , between two symmetric monoidal categories,  to their associated category of symmetric sequences and of operads.
					Given a Quillen equivalence $\lambda: \mathcal{C}=Ch_{\geq t}(\Bbbk) \rightleftarrows \mathcal{D}: R,$ so that  $\mathcal{D}$ is any symmetric monoidal category and the adjoint pair $(\lambda, R)$ is weak monoidal, we prove that the categories of connected operads $Op_\mathcal{C}$ and $Op_\mathcal{D}$ are Quillen equivalent. 
					
				\end{Abstract}
					AMS Classification numbers.  Primary: 18D50, 55U35  ;  Secondary:  55U15, 18D10
				\begin{Key words}
					Model category, coalgebra, cosimplicial frame, operad.
				\end{Key words}
				\section*{Introduction}
				This paper is inspired by the work of Schwede-Shipley who studied conditions for extending a Quillen equivalence  $\lambda: \mathcal{C} \rightleftarrows \mathcal{D}: R$ between two symmetric monoidal model categories to a Quillen equivalences on their associated sub-categories of monoid (see  in \cite{SS03}). Though ,  under good assumptions, an operad over a category  is a monoid in the associated category of symmetric sequences with the circle product,  the study of Schwede-Shipley still does not apply for operads since the circle product is not symmetric and does not distribute over colimits. 
				
				Muro extended this in \cite[Thm 1.1]{Muro} to show that, under favorable situations, a Quillen equivalence also induces a Quillen equivalence on the associated categories of non-symmetric operads. The goal of this paper is to provide an analogous result for symmetric operads in the particular case where $\mathcal{C}=Ch_{\geq t}(\Bbbk),$ $t\in \mathbb{N}\cup\{-\infty\},$ is the category of $t$-below truncated chain complexes.
				
	 Our method consists of constructing fundamental relations  between the adjoint pair $\lambda: \mathcal{C} \rightleftarrows \mathcal{D}: R$ and the corresponding pair  at the level of operads: $L: Op_{\mathcal{C}} \rightleftarrows Op_{\mathcal{D}}: \overline{R}.$

				We give, using cofree cocommutative coalgebras, an explicit  construction  for cosimplicial frames of operads. We use this construction to prove that  the realization of simplicial operads defines a functor which is weakly equivalent to the realization in each arity in  the underlying category $\mathcal{C}.$
		 This is expressed in  the following theorem (see Thm \ref{Equivalence_between_realization_functors}):
				\begin{Theorem A}
					Let $P_\bullet$ be a simplicial operad on chain complexes $\mathcal{C}=Ch_{\geq t}(\Bbbk),$ $t\in \mathbb{N}\cup\{-\infty\}.$   Then for any integer $r\geq 0, $ there is a zig-zag of quasi-isomorphisms
					\begin{center}
						$\arrowvert P_\bullet(\underline{r})\arrowvert_\mathcal{C} \overset{\simeq}{\longleftarrow} \bullet \overset{\simeq}{\longrightarrow}  \arrowvert P_\bullet\arrowvert_{ Op_\mathcal{C}}(\underline{r})$
					\end{center}
				\end{Theorem A}
				Every weak monoidal Quillen pair $\lambda:\mathcal{C} \rightleftarrows \mathcal{D}: R$ lifts, under good hypothesis, to a weak monoidal Quillen pair 
				\begin{equation}\label{Equation Quillen equivalence symmetric sequence}
				\overline{\lambda}: [FinSet,\mathcal{C}] \rightleftarrows [FinSet, \mathcal{D}]: \overline{R}
				\end{equation} between their associated categories of symmetric sequences, and the equivalence $(\ref{Equation Quillen equivalence symmetric sequence})$ lifts to an adjoint pair $L: Op_{\mathcal{C}} \rightleftarrows Op_{\mathcal{D}}: \overline{R}.$
				
				Now given any cofibrant operad, we consider its simplicial resolution to prove  the next proposition (see Prop \ref{Relation_betwen_left_adjoint_functors}) which gives a relation between the two  functors $\overline{\lambda}$ and $L.$
				\begin{Proposition A-1}
					Let $\lambda: \mathcal{C}=Ch_{\geq t}(\Bbbk) \rightleftarrows \mathcal{D}: R$ be a weak monoidal Quillen pair between the category $(Ch_{\geq t}(\Bbbk), \otimes, \Bbbk),$  $t\in \mathbb{N}\cup\{-\infty\},$  and any other symmetric monoidal category $(\mathcal{D}, \wedge, \mathbb{I}_\mathcal{D}).$
					If $P$ is a cofibrant operad in $Ch_{\geq t}(\Bbbk),$ then the morphism $\overline{\lambda}(U(P))\longrightarrow UL(P)$
					, which is adjoint to the  unit $\eta: P \longrightarrow \overline{R}L(P),$ 
					is a weak equivalence.	
				\end{Proposition A-1}
				This later proposition is the key ingredient to prove the main result of this paper (see  Theorem \ref{Quillen_equivalence_operads}).
				
				\begin{Theorem B} Let $(\mathcal{D}, \wedge, \mathbb{I}_\mathcal{D})$ be a cofibrantly generated symmetric monoidal model category with cofibrant unit $\mathbb{I}_\mathcal{D},$ and let  
					 $\lambda: \mathcal{C}=Ch_{\geq t}(\Bbbk)\rightleftarrows \mathcal{D}: R$ be a weak monoidal Quillen pair.  If the pair $(\lambda, R)$ is a Quillen equivalence, then the pair $L: Op_\mathcal{C} \rightleftarrows Op_\mathcal{D}: \overline{R}$ is a Quillen equivalence between the semi-model categories  $Op_\mathcal{C}$ and $Op_\mathcal{D}.$
					 
					 In addition, if the semi-model structure on $Op_\mathcal{D}$ is a strict model structure, then the pair $(L,\overline{R})$ is a strict Quillen equivalence.
				\end{Theorem B}
				
			Many other authors have tried to extend Quillen equivalences between two monoidal model categories to their associated category of symmetric operads in the sense of  Theorem B. For instance, in \cite[Thm 5.8]{hackney2017shrinkability} and \cite[$§$5.5]{white2016homotopical} the authors use  in their approach "pasting schemes" and "graphs". In these two cases, the authors consider  that the  categories  $\mathcal{C} $ and $\mathcal{D}$ are closed symmetric monoidal, with some additional conditions on the category $\mathcal{D}.$ In our constructions, the category $\mathcal{D}$  does not have a particular restriction.
			
			Our Theorem B is a particular case of \cite[Thm 8.10]{Pavlov_Jakob} who addresses colored operads and with slightly the same hypothesis. Our approach using commutative coalgebras is new and the tools we develop can be used in many other places.

		%	When this paper was written, the author learned that the  main result (Theorem B) overlaps with \cite[Thm 5.8]{hackney2017shrinkability} and \cite[$§$5.5]{white2016homotopical}. We have a completely different approach as they use "pasting schemes" and "graphs". In the two cases, these authors consider  that the  categories  $\mathcal{C} $ and $\mathcal{D}$ are closed symmetric monoidal, with some additional conditions on the category $\mathcal{D}.$  The second reference is  more general as it addresses  colored  operads. In our constructions, the category $\mathcal{D}$  does not have a particular restriction.

				\subsection*{Outline of the paper}
				In section 1,  we remind the background on chain complexes, weak monoidal adjoint pair and symmetric sequences. We give a general construction for cofree cocommutative coalgebras associated to chain complexes. We also prove that the cofree coalgebra of an acyclic chain complex is trivial.
				In section 2, we study the extension of weak monoidal Quillen pair and Quillen equivalence, between two symmetric monoidal model categories,  to the associated category of symmetric sequences.  In section 3, we define the cosimplicial frame for the realization functor of simplicial operads, and we prove theorem A. We end this section by giving a simplicial resolution of any arbitrary operad of chain complexes.
				Section 4 is dedicated to the proof of theorem B.	
				\begin{Acknowledgements}
					The research in this paper is supported by the OpLoTo grant and the Hubert Curien grant.
					The author is indebted to Benoit Fresse for numerous helpful discussions, suggestions and corrections during this work.
				\end{Acknowledgements}

				\section{preliminaries}
				\subsection{Chain complexes}
				Let $\Bbbk$ be a field. In this note, we denote by $Ch_{}(\Bbbk)$ the category of  differential $\mathbb{Z}$- graded  chain complexes over $\Bbbk.$  This category has a symmetric monoidal structure. The tensor product of chain complexes  $\forall V,W \in Ch(\Bbbk),$ is given by
				\begin{center}
					 $(V\otimes W)_n:= \underset{p+q=n}{\oplus} V_p \otimes W_q$
				\end{center}
				with the differential such that: $\forall x\otimes y \in V_p \otimes W_q, $ $d(x\otimes y)= d(x)\otimes y + (-1)^p x\otimes d(y).$
				The switch morphism $T:V\otimes W \longrightarrow W\otimes V$ involves the Koszul sign: $T(x\otimes y)=(-1)^{pq}y\otimes x.$
				
				The unit of the monoid $-\otimes-,$ which we  denote abusively  $\Bbbk,$ is the chain complex having $\Bbbk$ in degree $0$ and is trivial in all other degrees.
				
				We denote by $Ch_{\geq t}(\Bbbk) ,$ the sub-category of  $Ch_\Bbbk $ which consist of $t$-below truncated chain complexes  with $t\in \mathbb{N}\cup\{-\infty\},$ where  $Ch_{\geq -\infty}(\Bbbk)=Ch(\Bbbk).$ The induced tensor product $-\otimes -$ endows the category $Ch_{\geq t}(\Bbbk)$ with a monoidal structure. 
				
					\subsubsection*{Model category structure on $Ch_{\geq t}(\Bbbk), t \in \mathbb{N}$}
					The category $Ch_{\geq t}(\Bbbk)$ is a cofibrantly generated model category (for instance see   ( \cite[II p. 4.11, Remark 5 ]{Quill67}, \cite[ Thm 7.2 ]{Dwyer95homotopytheories}): 
					\begin{enumerate}
						\item [-]weak equivalences are quasi-isomorphisms;
						\item[-] fibrations are the morphisms that are surjective in degree $>t;$  
						\item[-] cofibrations are monomorphism with degreewise projective cokernels.
					\end{enumerate}
					Since  $\Bbbk$ is a field, all objects are cofibrant and fibrant in this model category.
					
					\subsubsection*{Model category structure on $Ch(\Bbbk)$}
					The category $Ch(\Bbbk)$ is a cofibrantly generated model category (for instance see \cite[remark after Thm. 9.3.1]{Hovey_Palmieri}): 
					\begin{enumerate}
						\item [-]weak equivalences are quasi-isomorphisms;
						\item[-] fibrations are surjections and 
						\item[-] cofibrations are morphisms having the left lifting property with respect to trivial fibrations.
					\end{enumerate}
					In this model category, the cofibrations, that we do not describe explicitly, are in particular degreewise split injections.

			\subsection{Coalgebras on chain complexes}	
				
						\begin{definition}[Coalgebra]
							A coalgebra on $Ch_{\geq 0}(\Bbbk) $ is a cotriple $(C, \triangle, \varepsilon)$ with $C$ an object in $Ch_{\geq 0}(\Bbbk) ,$ $\triangle: C\longrightarrow C\otimes C$ a map called (diagonal or) comultiplication and $\varepsilon: C \longrightarrow \Bbbk$ a map called the augmentation or counit, which satisfies the coassociativity and counitary conditions (see \cite[Section 1]{Sweedler1969hopf}). The category of coalgebra is denoted $\mathcal{C A}.$
							
							A coalgebra  $(C, \triangle, \varepsilon)$ is cocommutative when the comultiplication $\triangle$ is co-symmetric. The category of cocommutative coalgebras is denoted $CDGC.$
							\end{definition}
				
						 There is a fundamental theorem of coalgebras due to Getzler-Goerss:
						 \begin{theorem} \cite[Corollary 1.6]{GetzlerGoerss}\label{Universal property of coalgebras}
						 	Every coalgebra in  $Ch_{\geq 0}(\Bbbk)$ is a filtered colimit of its finite dimensional sub-coalgebras.
						 \end{theorem}

				\subsection{Cofree cocommutative coalgebra generated by a chain complex}\label{Cofree_section}

		 In this section we construct the cofree cocommutative differential non-negatively graded coalgebra generated by a chain complex $V \in Ch_{\geq 0}(\Bbbk).$ We prove that it sends acyclic chain complexes to contractible commutative differential graded coalgebra( CDGC). More precisely we will prove  that 
		 \begin{enumerate}
		 	\item [(i)] the forgetful functor $U: CDGC \longrightarrow Ch_{\geq 0}(\Bbbk)$ admits a right adjoint $C(-);$
		 	\item [(ii)] if $V \in Ch_{\geq 0}(\Bbbk)$ is such that $H(V)=0$ then the natural map $C(V) \longrightarrow \Bbbk$ is  a quasi-isomorphism.
		 \end{enumerate}
		 The dual version of this result is classical \cite[Chap 3: Example 6]{FHT01}, namely we have the adjunction 
		 \begin{equation}\label{Equation Adjunction between CDGA and Ch}
		 \Lambda: Ch_{\leq 0} (\Bbbk)   \rightleftarrows   CDGA: U
		 \end{equation}
		 between the category of non-positively graded chain complexes and the category of commutative  non-positively differential graded  algebras(CDGA) over $\Bbbk,$
		 where $\Lambda$ is the graded  commutative algebra functor .

		 A way to prove $(i)$ would be to take the linear dual of the adjunction (\ref{Equation Adjunction between CDGA and Ch}) but the problem here is that the linear dual of a  CDGA is a CDGC only when the algebra is of finite dimension. Therefore instead of looking at CDGA's we will consider profinite CDGA's and we will prove that this category is equivalent to CDGC. A key argument being the fact that every CDGC is a colimit of its finite dimensional subcoalgebras. Our proof was inspired by  \cite[Prop 1.7]{GetzlerGoerss}.
		 
		 We remind the definition of pro-objects as it appears in \cite[def 2.1]{DanielC2002}
		 \begin{definition}[Pro-objects]\label{Definition Pro Objects}
		 	\begin{enumerate}
		 		\item A left filtered category $\mathcal{J}$ is a non empty category satisfying the following conditions:
		 		\begin{enumerate}
		 			\item for any pair of objects $c_1, c_2 \in \mathcal{J},$ there exists an object $c_3\in \mathcal{J}$ and morphisms  $c_3 \longrightarrow c_1$ and $c_3 \longrightarrow c_2.$
		 			\item For every pair of morphisms $f,g: c_1 \longrightarrow c_2 $ in $\mathcal{J},$  there exists a morphism $h:c_0 \longrightarrow c_1$ such that $fh=gh.$
		 		\end{enumerate}
		 		\item A category $\mathcal{J}$ is  right filtered when  $\mathcal{J}^{op}$ is left filtered.
		 		\item For a category $\mathcal{C},$ the category of pro-objects on $\mathcal{C}$ and denoted $\text{pro-}\mathcal{C}$ is defined as follows:
		 		\begin{enumerate}
		 			\item Objects are  diagrams $X: \mathcal{J} \longrightarrow \mathcal{C},$ where $\mathcal{J}$ is a left filtered diagram.
		 			\item Given two pro-objects $X: \mathcal{J} \longrightarrow \mathcal{C}$ and $Y: \mathcal{K} \longrightarrow \mathcal{C},$ a morphism of pro-objects from $X$ to $Y$ is an element in the set
		 			\begin{center}
		 				$Hom_{\text{pro-}\mathcal{C}}(X, Y)= \underset{k \in \mathcal{K}}{\text{ lim }}  \underset{j \in \mathcal{J}}{\text{ colim }} Hom_\mathcal{C}(X_j, Y_k)$
		 			\end{center}
		 		\end{enumerate}
		 	\end{enumerate}
		 \end{definition}
		 
		 \begin{definition}[Profinite CDGA]
		 	Let $CDGA_f$ be the category of  non positively graded finite dimensional commutative $\Bbbk$-algebras.
		 	The category $\text{pro-}CDGA_f$ is called the category of profinite CDGA's..
		 \end{definition}

		 \newpage
		 
		 \begin{mylemma}\label{lemma hom commutes with filtered colimits}
		 	Let $\mathcal{C}$ be any of the categories CDGC or $Ch_{\geq \Bbbk}.$
		 	Let $\{C_\alpha\}$ be a right filtered diagram in $\mathcal{C}$ and let $D$ be a finite dimensional object in $\mathcal{C}.$ Then the natural map
		 	\begin{center}
		 		$\gamma:	\underset{\alpha}{ \text{ colim }} Hom_{ \mathcal{C}}(D, C_\alpha) \longrightarrow Hom_{\mathcal{C}}(D, \underset{\alpha}{ \text{ colim }} C_\alpha) $
		 	\end{center}
		 	is an isomorphism.
		 \end{mylemma}	
		 \begin{proof} We only make the proof in the case $\mathcal{C}=CDGC$ as this proof will include the chain complexes case.
		 	The map  $\gamma$ is injective using the universal property of colimits. 
		 	
		 	To prove that $\gamma$ is surjective, we consider a morphism of cocommutative coalgebras $f: D \longrightarrow C=\underset{\alpha}{ \text{ colim }} C_\alpha,$ this map factors as a morphism of chain complexes $D \longrightarrow C_\alpha$ since $D$ is finite. Since $f$ is a coalgebra morphism, we then have the following diagram
		 	\begin{center}
		 		$\xymatrix{ D \ar[r]^{}\ar[d] & C_\alpha \ar[r]& C= \underset{\alpha}{\text{colim }} C_\alpha\ar[d]\\
		 		D\otimes D \ar[r]& C_\alpha\otimes C_\alpha \ar[r]& C\otimes C
		 		}$
		 	\end{center}
		 On the other hand, since $-\otimes-$ preserves filtered colimits we have   $C\otimes C	= \underset{\alpha}{\text{colim }} C_\alpha \otimes C_\alpha.$ We then use the fact that $D$ is finite to deduce that there is a morphism $C_\alpha \longrightarrow C_\beta$ so that the following diagram commutes
		 	\begin{center}
		 	$\xymatrix{ D \ar[r]^{}\ar[d] & C_\alpha \ar[r] &C_\beta \ar[r]\ar[d] & C\ar[d]\\
		 		D\otimes D \ar[r]& C_\alpha\otimes C_\alpha \ar[r] & C_\beta\otimes C_\beta \ar[r] & C\otimes C
		 	}$
		 \end{center}
		This proves that $D \longrightarrow C_\alpha \longrightarrow C_\beta$ is a morphism of coalgebras.
		 We then deduce that $\gamma$ is surjective.					
		 \end{proof}

		 Since the dual of a finite dimensional commutative algebra is a cocommutative coalgebra, one deduces  the following definition:
		 \begin{definition}[Dual of profinite algebra]\label{Dual of profinite}
		 	We define the linear dual of profinite CDGA's as the contravariant functor 
		 	$D: \text{pro-}CDGA_f^{op} \longrightarrow CDGC$ such that:
		 	\begin{enumerate}
		 		\item For objects: given a  profinite commutative algebra $\{A_\alpha\},$ we set
		 		\begin{center}
		 			$D(\{A_\alpha\} ):= \underset{\alpha}{ \text{colim }} A_\alpha^{\vee},$
		 		\end{center} where $A_\alpha^{\vee}$ denotes the  dual of $A_\alpha;$
		 		\item For morphisms: given two profinite algebras $A=\{A_\alpha\}$ and $B=\{B_\beta\},$ we define the map of sets
		 		\begin{center}
		 			$Hom_{\text{pro-}CDGA_f}(A,B) \longrightarrow Hom_{CDGA}(D(B), D(A))$
		 		\end{center}
		 		as the composite
		 		\begin{align}
		 		\underset{\beta}{\text{ lim}}  	\underset{ \alpha}{\text{ colim }} Hom_{CDGA}(A_\alpha, B_\beta)&\cong \underset{\beta}{\text{ lim}}  	\underset{ \alpha}{\text{ colim }}Hom_{CDGC}(B_\beta^{\vee}, A_\alpha^{\vee})\\
		 		&	\overset{}{\cong}  \underset{\beta}{\text{ lim }}  	Hom_{CDGC}(B_\beta^{\vee}, \underset{ \alpha}{\text{ colim }} A_\alpha^{\vee})\\
		 		&	\cong  	Hom_{CDGC}(  \underset{\beta}{\text{ colim }}   B_\beta^{\vee}, \underset{ \alpha}{\text{ colim }} A_\alpha^{\vee})  				
		 		\end{align}	
		 		where the isomorphism  $(3)$ is provided by Lemma \ref{lemma hom commutes with filtered colimits}.
		 	\end{enumerate}
		 \end{definition}

		 \begin{proposition} \label{equivalence between coalgebra and profinite algebras}
		 	The linear dual functor described in Definition \ref{Dual of profinite} defines an anti-equivalence between the category $CDGC$  and the category $\text{pro-}CDGA_f.$
		 \end{proposition}
		 \begin{proof}
		 	This proposition  is similar to   \cite[Prop 1.7]{GetzlerGoerss} for associative coalgebras.
		 	
		 	The right inverse of the functor $D$  is given by the functor
		 	\begin{center}
		 		$D':  CDGC^{op} \longrightarrow \text{pro-}CDGA_f$ 
		 	\end{center} which is defined as follows: Let $C$ be a cocommutative coalgebra. Using Theorem \ref{Universal property of coalgebras}, one can write $C =\underset{\beta}{\text{colim }} C_\beta,$ where $C_\beta$ runs over finite dimensional sub cocommutative coalgebras of $C.$ We then set \begin{center}
		 	$D'(C):=\{C_\alpha^\vee\}$
		 \end{center}
		 where $C_\alpha^\vee$ is the dual of $C_\alpha.$	
		 
		 To complete the proof, one make the following computation: let $\widehat{A}=\{A_\alpha\}$ be an object of $\text{pro-}CDGA_f$  and  $C=\underset{\alpha}{\text{colim }} C_\alpha$ be a cocommutative coalgebra, where  $C_\alpha$ runs over finite dimensional sub cocommutative coalgebras of $C.$
		 \begin{align*}
		 Hom_{CDGC}(C,D(\{A_\alpha\}))&\cong \underset{\beta}{\text{ lim }}Hom_{CDGC} (C_\beta,  \underset{\alpha}{ \text{colim }} A_\alpha^{\vee})\\
		 & \cong \underset{\beta}{\text{ lim }} \underset{\alpha}{ \text{colim }} Hom_{ CDGC} (C_\beta,   A_\alpha^{\vee})\\
		 &\cong \underset{\beta}{\text{ lim }} \underset{\alpha}{ \text{colim }} Hom_{CDGA} (  A_\alpha , C_{\beta}^{\vee} )\\
		 & = Hom_{ \text{pro-}CDGA_f}(\{A_\alpha\}, D'(C) )
		 \end{align*}
		 In addition there is a natural isomorphism $C \overset{\cong}{\longrightarrow} DD'(C)$ justified by the following diagram
		 \begin{center}
		 	$\xymatrix{ C \ar[r] \ar[rd]_{\cong}& DD'(C) = \underset{\alpha}{\text{ colim }} C_\alpha^{\vee \vee}\ar[d]^{\cong}\\
		 		& \underset{\alpha}{\text{ colim }} C_\alpha
		 	}$
		 \end{center}
	 
	 To define explicitly and prove the other isomorphism $\varrho: \widehat{A} \overset{\cong}{\longrightarrow} D'D(\widehat{A}),$ we remind the following fact:
	 \begin{enumerate}
	 	\item [-] Any inclusion $D_\beta \hookrightarrow D(\widehat{A})= \underset{\alpha}{ \text{colim }} A_\alpha^{\vee}$ factors in the category of coalgebras as:
	 	\begin{center}
	 		$\xymatrix{ D_\beta \ar@{^{(}->}[rr] \ar@{^{(}->}[rd]_{j_*}& &\underset{\alpha}{ \text{colim }} A_\alpha^{\vee}\\
	 			& A_\delta^{\vee}\ar[ru]&
	 		}$
	 	\end{center}
 		\item [-] $D_\alpha:= \text{Im}( A^\vee_\alpha \longrightarrow D(\widehat{A}))$ is a sub-coalgebra of $D(\widehat{A})$ of finite dimension.
	 \end{enumerate}
	 
	We define the morphism of pro-objects  $\varrho: \widehat{A} \overset{}{\longrightarrow} D'D(\widehat{A})=\{ D_\beta\}$ as given by the maps $A_\delta \overset{j^*}{\longrightarrow} D_\beta^\vee$ of algebras dual of the morphisms of coalgebras $D_\beta \overset{j_*}{\longrightarrow} A_\delta^{\vee}.$ To prove that $\varrho$ is an isomorphism, we define the inverse map 
	$\varrho': D'D(\widehat{A})=\{ D_\beta\} \overset{}{\longrightarrow} \widehat{A} $ as given by the maps of the form $D^\vee_\alpha= ( \text{Im}( A^\vee_\alpha \longrightarrow D(\widehat{A})))^\vee \overset{k^*}{\longrightarrow} A_\alpha$ which is dual of natural maps $A_\alpha^\vee \overset{k_*}{\longrightarrow} \text{Im}( A^\vee_\alpha \longrightarrow D(\widehat{A})).$
	
Now to prove that $\varrho$ and $\varrho'$ are truly isomorphism  inverse, it will be enough to show that $Hom_{pro-CDGA_f}(\varrho, P)$ and  $Hom_{pro-CDGA_f}(\varrho', P),$ $P\in CDGA,$ are isomorphism inverse. For this later claim, we make the following computation:

\begin{align*}
\underset{\alpha}{\text{colim }} Hom(A_\alpha, P) &\overset{Hom(\varrho', P)}{\longrightarrow} \underset{\beta}{\text{colim }} Hom(D^\vee_\beta, P) & \overset{Hom(\varrho, P)}{\longrightarrow} \underset{\alpha}{\text{colim }} Hom(A_\alpha, P) \\
[A_\alpha\overset{f}{\longrightarrow}P]& \mapsto [D^\vee_\alpha \overset{k^*}{\longrightarrow} A_\alpha\overset{f}{\longrightarrow}P ] \longmapsto  &[A_\delta \overset{j^*}{\longrightarrow}D^\vee_\alpha \overset{k^*}{\longrightarrow} A_\alpha\overset{f}{\longrightarrow}P ] =[A_\alpha\overset{f}{\longrightarrow}P];
\end{align*}
which proves that 	$Hom_{pro-CDGA_f}(\varrho, P)\circ Hom_{pro-CDGA_f}(\varrho', P)=Id.$

On the other hand, we have the composite:
	\begin{align*}
	\underset{\beta}{\text{colim }} Hom(D^\vee_\beta, P) &\overset{Hom(\varrho, P)}{\longrightarrow}  \underset{\alpha}{\text{colim }} Hom(A_\alpha, P)& \overset{Hom(\varrho', P)}{\longrightarrow} \underset{\beta}{\text{colim }} Hom(D^\vee_\beta, P)        \\
	[D^\vee_\beta\overset{g}{\longrightarrow}P]& \mapsto [A_\delta \overset{j^*}{\longrightarrow}    D^\vee_\beta\overset{g}{\longrightarrow}P] \longmapsto  &[D^\vee_\delta \overset{k^*}{\longrightarrow}A_\delta \overset{j^*}{\longrightarrow}    D^\vee_\beta\overset{g}{\longrightarrow}P] =[D^\vee_\beta\overset{g}{\longrightarrow}P];
	\end{align*}
	which proves that 	$Hom_{pro-CDGA_f}(\varrho', P)\circ Hom_{pro-CDGA_f}(\varrho, P)=Id.$
	 
		\end{proof}

			The consequences of Proposition \ref{equivalence between coalgebra and profinite algebras}   is the next proposition which gives the cofree construction of cocommutative coalgebras. Note that the forgetful functor from $CDGC$ to $Ch_{\geq 0}$ 
			 makes all colimits in $CDGC$ (see \cite[1.8.]{GetzlerGoerss}).
			
			\begin{theorem}\label{Cofree Cocommutative_Coalgebra}
				The forgetful functor $U: CDGC \longrightarrow Ch_{\geq 0}(\Bbbk) $ has a right adjoint $C:  Ch_{\geq 0}(\Bbbk)  \longrightarrow  CDGC.$
			\end{theorem}
		
	We now  make a short detour to prove the next lemma that will  be used to prove  Theorem \ref{Cofree Cocommutative_Coalgebra}.

		\begin{mylemma}\label{Lemma Colim with profinite algebras}
			Let $W$ be a non-positively graded chain complex and let $\{\Lambda W /I \}$ be the profinite algebra where $I$ runs over all differential  ideals of $\Lambda W$ with finite codimension. Given any finite dimensional algebra $A,$ the natural map
			\begin{center}
				$\Gamma:  \underset{I}{\text{ colim }} Hom_{CDGA}( \Lambda W /I , A) \overset{}{\longrightarrow} Hom_{CDGA}( \Lambda W  , A),$
			\end{center}
			 induced by the quotient maps $  \Lambda W \longrightarrow   \Lambda W/I,$ is a bijection.
		\end{mylemma}
		\begin{proof}
			To define the inverse $\Gamma'$ of $\Gamma,$ we make first the following observation:
			
		given any algebra map $f:  \Lambda W \longrightarrow A,$ we form the following short exact sequence(of chain complexes) \begin{center}
			$0\longrightarrow \text{Ker}f \longrightarrow \Lambda W \longrightarrow \text{Im}f \longrightarrow 0.$
		\end{center}  Since $A$ has a finite dimension and $\text{Im}f \subset A,$ one deduces that $\text{Ker}f$ is a finite co-dimensional ideal of $ \Lambda W.$ 
		
		Using this observation we define the map 	\begin{center}
			$\Gamma': Hom_{CDGA}( \Lambda W  , A) \longrightarrow \underset{I}{\text{ colim }} Hom_{CDGA}( \Lambda W /I , A) $
		\end{center}
		by
		\begin{center}
			$\Gamma'_I(f):=[  \Lambda W/\text{Ker}f\overset{f}{ \longrightarrow} A ]$  
		\end{center}	
	and we make the following computation

			\begin{align*}
			\Gamma \Gamma'(  f:  \Lambda W \longrightarrow A)&= \Gamma( [  \Lambda W/\text{Ker}f \longrightarrow A ] )\\
			& = [\Lambda W\longrightarrow \Lambda W/\text{Ker}f \longrightarrow A]\\
			&= [ f:  \Lambda W \longrightarrow A]
			\end{align*}
			
			Conversely, let $[g: \Lambda W/I \longrightarrow A]$ be the class in  $\underset{I}{\text{ colim }} Hom_{CDGA}( \Lambda W /I , A) $ represented by $g: \Lambda W/I \longrightarrow A.$  We have: 
			\begin{align}
			\Gamma'\Gamma([g: \Lambda W/I \longrightarrow A]) &= \Gamma'(\widetilde{g}:  \Lambda W\longrightarrow \Lambda W/I \overset{g}{\longrightarrow} A)\\
			& = [ \Lambda W/\text{Ker}\widetilde{g} \overset{\widetilde{g}}{\longrightarrow} A]\\
			&\overset{}{=}[g: \Lambda W/I \longrightarrow A]
			\end{align}
			where $(7)$	is justified by the fact that $ I \subseteq \text{Ker}\widetilde{g}. $

		\end{proof}

		\begin{Proof of Theorem Cofree coalgebras}
			
			Let	$V$ be a finite dimensional chain complex. We denote by  $C(V)$  the continuous dual of the profinite completion of the symmetric algebra $\Lambda V^\vee.$  Namely
			\begin{align*}
			C(V)&=D(\{ \Lambda V^\vee /I\})\\
			&=  \underset{I}{\text{ colim }} (\Lambda V^\vee /I)^\vee
			\end{align*}
			where $I$ runs over all differential ideals of $\Lambda V^\vee$ of finite co-dimension.
			If $I\subset J$ are two such ideals in  $\Lambda V^\vee,$ then their induced map in this diagram is the dual of the quotient $\Lambda V^\vee /I \longrightarrow \Lambda V^\vee /J.$
			
			Let $C=\underset{\beta}{\text{ colim }} C_\beta$ be a cocommutative coalgebra written as its finite sub-coalgebras. We make the following computation
			\begin{align}
			Hom_{ CDGC}(C, C(V))&\overset{}{\cong} Hom_{ \text{pro-}CDGA_f }(\{ \Lambda V^\vee /I \}, D'(C))\\
			&\overset{}{\cong} \underset{\beta}{\text{ lim }}  \underset{ I}{\text{ colim }} Hom_{ CDGA }( \Lambda V^\vee/I , C^\vee_\beta)\\
			&\overset{}{\cong} \underset{\beta}{\text{ lim }}   Hom_{ CDGA }( \Lambda V^\vee , C^\vee_\beta)\\
			&\cong \underset{\beta}{\text{ lim }}   Hom_{ Ch_{\leq 0}(\Bbbk)}(  V^\vee , UC^\vee_\beta)\\
			&\cong \underset{\beta}{\text{ lim }}   Hom_{ Ch_{\geq 0}(\Bbbk)}(    UC_\beta, V)\\
			&	\cong   Hom_{ Ch_{\geq 0}(\Bbbk)}(    UC, V)
			\end{align}
			where
			\begin{enumerate}
				\item [-]$(8)$ is induced by Proposition \ref{equivalence between coalgebra and profinite algebras} and $D'(C)=\{C^{\vee}_\beta\};$
				\item[-] The isomorphism $(10)$ is given by Lemma \ref{Lemma Colim with profinite algebras}.
			\end{enumerate}

			For a general $V, $ we define 	\begin{center}
				$C(V):= \underset{\alpha}{\text{ colim }} C(V_\alpha)$
			\end{center}
			where $V_\alpha$ runs over finite dimensional sub chain complexes of $V.$ We  have
			\begin{align}
			Hom_{CDGC}(C, C(V))&\cong \underset{\beta}{ \text{lim }}Hom_{ CDGC}(C_\beta, C(V))\\
			&\cong \underset{\beta}{ \text{lim }}  \underset{\alpha}{ \text{colim }}   Hom_{  CDGC}(C_\beta, C(V_\alpha))\\
			&\cong \underset{\beta}{ \text{lim }}  \underset{\alpha}{ \text{colim }}   Hom_{  Ch_{\geq 0}(\Bbbk)}(UC_\beta, V_\alpha)\\
			& \cong Hom_{  Ch_{\geq 0}(\Bbbk)}(UC, V)
			\end{align}
			where $(16)$ follows from Lemma $5.$
		\end{Proof of Theorem Cofree coalgebras}

				\begin{theorem}\label{C(V)_contractible_lemma}
					We consider that $\Bbbk$ is a field of characteristic $0.$ If $V$ is an acyclic chain complex (i.e. $V\simeq 0$), then the natural map
					\begin{center}
						$C(V)\longrightarrow \Bbbk$
					\end{center} 
					is a quasi isomorphism.
				\end{theorem}
				\begin{proof}
					We have $V= \underset{\alpha}{\text{ colim }} V_\alpha,$ where $V_\alpha$ are finite and acyclic subcomplexes of $V.$ We use this system to write 
			 $C(V)= \underset{\alpha}{\text{ colim }} C(V_\alpha)$  (see the proof of Thm \ref{Cofree Cocommutative_Coalgebra}).

					runs over finite subcomplexes of $V.$ Then the proof reduces to proving that $C(V)\overset{\simeq}{\longrightarrow} \Bbbk$ when $V$  is a finite and acyclic chain complex. 
					On the other hand, a finite and acyclic chain complex is of the form:
					$V=D(n_1)\times ...\times D(n_k)$ where $\forall n, D(n)$ is the chain complex  such that 
					
					\begin{center}
						\[ D(n)_i:= \left\{ \begin{array}{ll}
						\Bbbk  & \mbox{if $i=n$ or $i=n+1$},\\
						0 & \mbox{ortherwise},\end{array} \right. \] 
						
						with the identity $Id: \Bbbk  \longrightarrow \Bbbk $ as  the differential $d: D(n)_{n+1}\longrightarrow D(n)_n.$ 
					\end{center} 
					Since the functor $C$ is a right adjoint functor (proved in Thm \ref{Cofree Cocommutative_Coalgebra}), one has 
					\begin{center}
						$C(V)= C(D(n_1))\times ... \times C(D(n_k))$
					\end{center}
					Therefore the proof reduces to proving that $C(D(n))\overset{\simeq}{\longrightarrow} \Bbbk,$ for any $n.$ In what follows, before proving this result we give first an explicit description of $C(D(n)).$
					\begin{enumerate}
						\item When $n \neq 0,$ we have
						\begin{center}
							 $C(D(n))= \underset{I}{\text{ colim }} (\Lambda(x,y)/I)^\vee,$ 
							 
							 where $|x|=-n, |y|=-n-1$ and $dx=y.$
						\end{center}
						Let $I_m$ be the ideal of $\Lambda(x,y)$ which consists of elements of degree less that $-m.$ Since the system $\Lambda(x,y)/I_m,  m\geq 0,$ is cofinal among the quotients of $\Lambda(x,y)$ of finite dimension, one deduces that 
						\begin{align*}
						C(D(n))=& \underset{m}{\text{ colim }} (\Lambda(x,y)/I_m)^\vee\cong \Lambda(x^*, y^*),
						\end{align*} 
						where $|x^*|=n, |y^*|=n+1$ and $dy^*=x^*.$
						This is clearly an acyclic chain complex.
						\item When $n = 0,$ we still have
						\begin{center}
							$C(D(n))= \underset{I}{\text{ colim }} (\Lambda(x,y)/I)^\vee,$ 
							 with $|x|=0, |y|=-1$ and $dx=y.$
						\end{center}
						Given a monic polynomial $q(x),$  let $(q(x), dq(x))$ be the ideal of $\Lambda(x,y)$ generated by $ \{q(x), dq(x)\}.$ 
						One can see that the system given by the quotients 
						\begin{center}
								$\Lambda(x,y)/( q(x), dq(x))$
						\end{center}
					 is cofinal among the quotients of $\Lambda(x,y)$ of finite dimension. 
						We can then deduce as in the previous case that 
						\begin{align*}
						C(D(0))& = \underset{q(x)}{\text{ colim }} (\Lambda(x,y)/(q(x),dq(x)))^\vee\\
						& \cong \underset{q(x),\ \  x \text{ divides } q(x)}{\text{ colim }} (\Lambda(x,y)/(q(x),dq(x)))^\vee.
						\end{align*}
						At this point, since the colim used in this expression is filtering,  to show that $C(D(0))$ is contractible it is sufficient to prove that each complex $\Lambda(x,y)/(q(x), dq(x))$ is contractible.
						
						Let $n=degree(q(x)).$ We consider the following morphism 
						\begin{center}
							$\Phi: \Lambda(x,y)/(q(x), dq(x)) \longrightarrow \Bbbk1 \oplus ...\oplus \Bbbk x^{n-1} \oplus \Bbbk y \oplus ... \oplus \Bbbk x^{n-2}y$
						\end{center}
						which associates to each class $[\alpha]$ presented by $\alpha=u(x)+yv(x),$ the polynomial $\Phi([\alpha]):=u_1(x)+yv_1(x),$ where $u_1(x)$ (resp. $v_1(x)$) is the rest of Euclidean division of $u(x)$ (resp. $v(x)$) by $q(x)$ (resp. $q'(x)$). This morphism is clearly a bijection.
						 
						 On the other hand, since one have the  differential equation $d(y)=x,$ the projection 
						 \begin{center}
						 	$\Bbbk1 \oplus ...\oplus \Bbbk x^{n-1} \oplus \Bbbk y \oplus ... \oplus \Bbbk x^{n-2}y \longrightarrow \Bbbk1= \Bbbk$
						 \end{center}
						is a quasi-isomorphism. This completes the proof.
								
					\end{enumerate}
				\end{proof}

				\subsection{Weak  monoidal adjoint pair}
				\begin{definition}[Quillen pair]\begin{enumerate}
						\item [(1)] 	A pair of adjoint functor \begin{center}
							$\lambda: \mathcal{C}  \rightleftarrows \mathcal{D}: R$
						\end{center}
						between two model categories is a $\emph{Quillen pair}$ if the right adjoint $R$ preserves fibrations and trivial fibrations.
						
						A Quillen pair $(\lambda, R)$ induces a derived adjoint pair $(\mathbb{L}\lambda, \mathbb{L}R)$ of functors between the homotopy categories of $\mathcal{C}$ and $\mathcal{D}$(see  \cite[I.4.5]{Quill67}). 
						\item [(2)] A Quillen pair is a $\emph{Quillen equivalence}$ if the associated derived functors are  equivalences between the homotopy categories.
					\end{enumerate}
					
				\end{definition}
				
				\begin{definition}[Weak monoidal]

					A $\emph{weak monoidal Quillen pair}$ between two monoidal model categories $(\mathcal{C}, \wedge, \mathbb{I}_{\mathcal{C}})$ and $(\mathcal{D}, \wedge, \mathbb{I}_{\mathcal{D}})$ consists of a Quillen pair  $\lambda: \mathcal{C}  \rightleftarrows \mathcal{D}: R$  satisfying the following conditions:
					\begin{enumerate}
						\item [a)] For any pair of objects $X $ and $Y$ in $\mathcal{D},$ there is a  morphism 
						\begin{center}
							$\varphi_{X,Y}: R(X)\wedge R(Y)\longrightarrow R(X\wedge Y),$ and a morphism $\nu: \mathbb{I}_{\mathcal{C}}\longrightarrow R(\mathbb{I}_{\mathcal{D}}),$ 
						\end{center}
						natural in $X , Y$ and coherently associative, commutative  and unital (see diagrams 6.27 and 6.28 of $\cite{Borc94}$).
						\item [b)] If $X$ and $Y$ are cofibrant, then the adjoint of the composition 
						
						$X\wedge Y\overset{\eta \wedge \eta }{\longrightarrow} R\lambda(X)\wedge R\lambda(Y)\overset{\varphi_{\lambda(X), \lambda(Y)}}{\longrightarrow} R(\lambda(X)\wedge \lambda(Y))$
						and 	denoted 
						\begin{center}
							$\widetilde{\varphi}_{X,Y}: \lambda (X\wedge Y)\longrightarrow \lambda(X)\wedge \lambda(Y)$
						\end{center}	
						is  a weak equivalence.
						\item [c)] for some (and hence any) cofibrant replacement $\mathbb{I}_{\mathcal{C}}^{c}\overset{\simeq}{\longrightarrow} \mathbb{I}_{\mathcal{C}}$ of the unit object in $\mathcal{C},$ the composite map 
						$ \lambda(\mathbb{I}_{\mathcal{C}}^{c})\overset{ }{\longrightarrow} \lambda(\mathbb{I}_{\mathcal{C}}) \overset{ \widetilde{\nu}}{\longrightarrow} \mathbb{I}_{\mathcal{D}}$
						is a weak equivalence in $\mathcal{D},$ where 	
						
						$\widetilde{\nu}:\lambda(\mathbb{I}_{\mathcal{C}})\longrightarrow \mathbb{I}_{\mathcal{D}} $ is the adjoint of $\nu.$
					\end{enumerate}
				\end{definition}
				
				In the literature (see \cite{SS03}), one says that the functor $R$ is \emph{lax symmetric monoidal} when condition $a)$ is satisfied. In this note we sometimes refer to maps $\varphi_{X,Y}$ as the lax monoidal structure morphisms associated to $R.$

				\subsection{Symmetric sequence}
				We give here the definition of a symmetric sequence along with the monoidal structure on the category of symmetric sequences. We refer to \cite{Ching12} for more details on this topic. 
				
				Let $(\mathcal{C}, \wedge, \mathbb{I}_{\mathcal{C}})$ be a symmetric monoidal category.
				
				\begin{definition}[Symmetric sequence]
					A symmetric sequence in the category $\mathcal{C}$ is a functor $M:FinSet\longrightarrow \mathcal{C}$ from the category $FinSet,$ whose objects are finite sets and whose morphisms are bijections, to $\mathcal{C}.$ Denote the category of all symmetric sequences in $\mathcal{C}$ by $[FinSet, \mathcal{C}]$(in which morphisms are natural transformations).
				\end{definition}
				
				Let $FinSet_0$ be the category whose objects are the finite sets $\underline{r}:=\{1, ..., r\}$ for $r\geq 0$ (with $\underline{0}$ the empty set), and whose morphisms are bijections. $FinSet_0$ is clearly a subcategory of $FinSet,$ and any symmetric sequence $M:FinSet\longrightarrow \mathcal{C}$ is determined, up to canonical isomorphism, by its restriction $M:FinSet_0\longrightarrow \mathcal{C}.$ This restriction consists of the sequence $M(\underline{0}), M(\underline{1}), M(\underline{2}), ...$ of objects in $\mathcal{C},$ together with an action of the symmetric group $\Sigma_r$ on $M(\underline{r}),$ hence the name "symmetric sequence."
				
				\begin{definition}
					For a finite set $J,$ we define the category $J/FinSet_0$ as follows. The class of objects of $J/FinSet_0$  consists of all functions (not necessary bijection) $f:J\longrightarrow I$ for some $I\in FinSet_0,$ and the set of morphisms from $(f:J\longrightarrow I)$ to $(f':J\longrightarrow I')$ is the set of bijections $\sigma: I\longrightarrow I'$ such that $f'=\sigma\circ f.$
				\end{definition}
				\begin{definition}
					Let $M$ and $N$ be two symmetric sequences in $\mathcal{C}.$ For each finite set $J,$ we define a functor \begin{center}
						$(M,N):J/FinSet_0\longrightarrow \mathcal{C}$
					\end{center}
					on objects by \begin{center}
						$(M,N)(f:J\longrightarrow I):=M(I)\wedge\underset{i\in I}{\bigwedge}N(f^{-1}(i)).$
					\end{center}
					For morphism $\sigma: I\longrightarrow I'$ in $J/FinSet_0$ we define 
					\begin{center}
						$(M,N)(\sigma):=M(I)\wedge\underset{i\in I}{\bigwedge}N(f^{-1}(i))\longrightarrow M(I')\wedge\underset{i'\in I'}{\bigwedge}N((\sigma f)^{-1}(i')).$
					\end{center}
					by combining map $M(\sigma)$ with the permutation of the smash product identifying the term corresponding to $i\in I$ with the term corresponding to $\sigma(i)\in I'.$ 
				\end{definition}
				\begin{definition}[Composition product]\label{composition_product}
					For symmetric sequences $M,N,$ we define a symmetric sequence $M\circ N$ by \begin{center}
						$ ( M\circ N)(J):=\underset{f\in J/FinSet_0}{colim}(M,N)(f).$
					\end{center}
					A bijection $\theta: J\longrightarrow J'$ determines a map $(M,N)(f)\longrightarrow (M,N)(f\circ \theta ^{-1})$ whose explicit as a map
					\begin{center}
						$ M(I)\wedge\underset{i\in I}{\bigwedge}N(f^{-1}(i))\longrightarrow M(I)\wedge\underset{i\in I}{\bigwedge}N(\theta(f^{-1}(I))) $
					\end{center}
				is given by  the identity on $M(I)$ and by the action of the bijections $\theta\rvert_{f^{-1}(i)}:f^{-1}(i)\overset{\cong}{\longrightarrow} \theta (f^{-1}(i))$ on the symmetric sequence $N.$
 We thus obtain induced maps 
					\begin{center}
						$(M\circ N)(\theta):= (M\circ N)(J)\longrightarrow (M\circ N)(J')$
					\end{center}
					that make $M\circ N$ into a symmetric sequence in $\mathcal{C}.$
				\end{definition}
				\begin{definition}[Unit symmetric sequence]
					The unit symmetric sequence $\mathbb{I}$ in the symmetric monoidal category $(\mathcal{C}, \wedge, \mathbb{I}_{\mathcal{C}})$ is given by \begin{center}
						$\mathbb{I}(J)=\mathbb{I}_{\mathcal{C}},$ if $|J|=1,$ and $\mathbb{I}(J)=0$ otherwise; 
					\end{center}
					where $0$ is the initial object  in $\mathcal{C}.$ The map $\mathbb{I}(J)\longrightarrow \mathbb{I}(J')$ induced by a bijection $J\longrightarrow J'$ is the identity morphism on $\mathbb{I}_{\mathcal{C}}$ or $0$ as appropriate.
				\end{definition}
				If the category $(\mathcal{C}, \wedge, \mathbb{I}_{\mathcal{C}})$ is such that $\wedge$ commutes with finite colimits, then the composition product $\circ$ is a monoidal product and $([FinSet, \mathcal{C}], \circ, \mathbb{I} ) $ is a monoidal category, but not symmetric (see \cite[Prop 2.9]{Ching05}). For instance the category $(Ch_{\geq t}(\Bbbk), \otimes, \Bbbk)$ is closed symmetric monoidal. This implies that the tensor product $\otimes$ has a right adjoint, so it preserves all colimits. Therefore $([FinSet, Ch_{\geq t}(\Bbbk)], \circ, \mathbb{I} ) $ is a monoidal category.
				
				In what follows, we equip the category of symmetric sequences $[FinSet, \mathcal{C}]$ (viewed as functor categories) with the projective model structure, where 
				\begin{enumerate}
					\item [-] the  fibrations and the weak equivalences are natural transformations that are objectwise such morphisms in $\mathcal{C}.$ 
				\end{enumerate}
				
				\section{Equivalence of symmetric sequences on Quillen equivalent categories}
				The goal of this section is to show that Quillen equivalences (resp. weak monoidal Quillen pair) between symmetric monoidal categories can be lifted to their associated categories of symmetric sequences.  More precisely,
				let $\lambda: \mathcal{C}\rightleftarrows \mathcal{D}: R $ be a Quillen pair between cofibrantly generated model categories. 
						One can always lift  the pair $(\lambda, R)$ to a Quillen pair
						\begin{center}
							$\overline{ \lambda}: [FinSet, \mathcal{C}]  \rightleftarrows[FinSet, \mathcal{D}]: \overline{R}$
						\end{center}
						preserving in this process some properties of the adjunction  $(\lambda, R)$ that we mention in the following proposition.
						
						\begin{proposition} \label{lift_adjonction_symmetric_sequence}We have the following results.
							\begin{enumerate}
								\item If $\lambda: \mathcal{C} \rightleftarrows \mathcal{D}: R $ is a Quillen equivalence, then so is $(  \overline{ \lambda}, \overline{R}).$
								\item If $(\mathcal{C}, \wedge, \mathbb{I}_{\mathcal{C}})$ and $(\mathcal{D}, \wedge, \mathbb{I}_{\mathcal{D}})$ are symmetric monoidal categories so that the categories $([FinSet, \mathcal{C}], \circ, \mathbb{I})$ and $([FinSet, \mathcal{D}], \circ, \mathbb{I})$ are monoidal categories,  and  if the pair $(\lambda, R)$ is a weak monoidal Quillen pair, then so is the pair $(\overline{\lambda},\overline{R} ).$
							\end{enumerate}
						\end{proposition}

		In order to give the proof of this result, we want to be certain that the right composition of a cofibration, in the category of symmetric sequences, with a cofibrant object is again a cofibration. This property we describe is itself a consequence of the pushout product axiom that we prove in the next lemma. We think it also appears in various places in the literature: in \cite[Thm 6.1 ]{Harper} when the based category $\mathcal{C}$ is closed symmetric monoidal, in \cite{Rezk} for the based category $\mathcal{C}=sSet.$
				
				\begin{mylemma}\label{Cofibration_circle_product}
					Let $(\mathcal{C}, \wedge, \mathbb{I}_{\mathcal{C}})$ be a cofibrantly generated symmetric monoidal model category . If $i: A\rightarrowtail  A'$ and $j: B\rightarrowtail  B'$ are cofibrations in $[Finset, \mathcal{C}]$ such that $B$ and $B'$ are cofibrant, then the natural morphism
					\begin{center}
						$(i_*, j_*): A'\circ B \underset{A\circ B}{\amalg} A\circ B'\longrightarrow A'\circ B'$
					\end{center}
					is a cofibration. In addition if $i$ or $j$ is acyclic, then so is $(i_*, j_*).$
					
				\end{mylemma}
				
				\begin{proof} We divided the proof is three steps. We consider the case of generating cofibration,  of cellular morphism, and finally of cofibration.  
					\begin{enumerate}
						\item Suppose that $K\otimes \mathbb{I}(\Sigma_r) \overset{i}{\rightarrowtail } K'\otimes \mathbb{I}(\Sigma_r) $ is constructed from a generating cofibration $K \rightarrowtail K'$ of $\mathcal{C}$ ,  and let $J$ be a finite set. We have  
						\begin{align*}
						(K\otimes \mathbb{I}(\Sigma_r)\circ B)(J)\cong & \underset{J_1\amalg ... \amalg J_r}{\bigvee} K\wedge B(J_1) \wedge ... \wedge B(J_r)\\
						\cong &K\wedge T^r_B(J)
						\end{align*}
						where,  $T^r_B(J)=  \underset{J_1\amalg ... \amalg J_r}{\bigvee} B(J_1) \wedge ... \wedge B(J_r).$
						
						Since the objects $B(J_i)$ and $B'(J_i)$ are cofibrant and  $\mathcal{C}$ is a symmetric monoidal model category, one deduces that $T_B^r(J) \longrightarrow T_{B'}^r(J)$ is a cofibration in $\mathcal{C}$ and is an acyclic cofibration if $j$ is. It then follows that the the induced morphism  \begin{center}
							$ K'\wedge T_{B}^r(J) \underset{K\wedge T_{B}^r(J)}{\amalg} K\wedge T_{B'}^r(J) \longrightarrow K'\wedge T_{B'}^r(J)$
						\end{center} is a cofibration in $\mathcal{C}$  and an acyclic cofibration if $i$ or $j$ is. This gives  the result in this specific case of $i.$  In addition, Since the functor $-\circ -$ preserves left colimits, this result generalizes to any cofibration of the form $\underset{\alpha}{\bigvee}K_\alpha\otimes \mathbb{I}(\Sigma_\alpha) \overset{i}{\rightarrowtail } \underset{\alpha}{\bigvee} K'_\alpha\otimes \mathbb{I}(\Sigma_\alpha).$ 
						
						\item We assume now that $i: A\longrightarrow A'$ is a cellular morphism presented by the $\beta$-sequence ($\beta$ is an ordinal) :\begin{center}
							$A=A<0> \longrightarrow A<1> \longrightarrow ... \longrightarrow A<\rho> \longrightarrow ... (\rho < \beta)$
						\end{center}
						where $A'= \underset{\rho}{\text{colim}} A<\rho>,$ with  pushout diagrams (cell attachment): 
						\begin{center}
							$\xymatrix{ K=\underset{\alpha}{\bigvee}K_\alpha\otimes \mathbb{I}(\Sigma_\alpha) \ar[r] \ar[d]& A<\rho> \ar[d]^{}\\
								K'=	\underset{\alpha}{\bigvee}K'_\alpha\otimes \mathbb{I}(\Sigma_\alpha) \ar[r]& A<\rho +1>
							}$
						\end{center}
						One deduces from this diagram the following pushout:
						\begin{center}
							$\xymatrix{K'\circ B \underset{K\circ B}{\amalg} K\circ B' \ar[r] \ar@{>->}[d]& A<\rho+1>\circ B \underset{A<\rho>\circ B}{\amalg} A<\rho> \circ B'  \ar[d]^{i_\rho}\\
								K'\circ B' \ar[r]& A<\rho+1> \circ B'
							}$
						\end{center}
						where the left vertical map is a cofibration (or acyclic cofibration) using $1.$ One deduces that the right vertical map is also a cofibration (or acyclic cofibration).

						We deduce that we have the composite 	
						\begin{center}
							$    \xymatrix{  A<\rho+1>\circ B \underset{A<\rho-1>\circ B}{\amalg} A<\rho-1> \circ B' \ar[d]^{\cong}& \\ 
								A<\rho+1>\circ B \underset{A<\rho>\circ B}{\amalg} A<\rho> \circ B\underset{A<\rho-1>\circ B}{\amalg} A<\rho-1> \circ B'    \ar@{>->}[d] ^{i_{\rho-1}}&\\
								A<\rho+1>\circ B \underset{A<\rho>\circ B}{\amalg} A<\rho> \circ B' \ar@{>->}[d]^{i_p} \\ 
								A<\rho+1> \circ B'&
							} $ 
						\end{center} 	
						We then deduce the result by induction on $\rho.$ 
						
						\item Finaly let us assume that $i$ is an arbitrary cofibration. The map $i$ is by definition the retract of a cellular map $X\overset{i'}{\longrightarrow }Y.$  Using the universal property of colimits, one deduces that $(i_*, j_*): A'\circ B \underset{A\circ B}{\amalg} A\circ B'\longrightarrow A'\circ B'$ is a retract of $Y\circ B \underset{X\circ B}{\amalg} X\circ B'\longrightarrow Y\circ B'$ which is a cofibration according to $2.$ Therefore we deduce that $(i_*, j_*)$ is also a cofibration, and is acyclic if one of $i$ or $j$ is.
					\end{enumerate}
				\end{proof}
				
				\begin{mylemma} \label{stable_cofibration_lemma}
					Let $(\mathcal{C}, \wedge, \mathbb{I}_{\mathcal{C}})$ be a symmetric monoidal model category, and cofibrantly generated. If $i: A\rightarrowtail  A'$ is a cofibration in $[Finset, \mathcal{C}]$ , and $B'$ is a cofibrant symmetric sequence in $\mathcal{C},$ then the natural morphism
					\begin{center}
						$  A\circ B'\longrightarrow A'\circ B'$
					\end{center}
					is a cofibration. 
				\end{mylemma}
				\begin{proof}
					We apply the result of Lemma \ref{Cofibration_circle_product} in specific case $j: \emptyset \rightarrowtail  B'.$
				\end{proof}

				\begin{Proof of proposition Lif adjunction symmetric sequences}

					\begin{enumerate}
						\item Let $C$ be a cofibrant object in $[Finset, \mathcal{C}] , D$ a fibrant object in $[Finset, \mathcal{D}] $ and a weak equivalence $ \overline{ \lambda}(C)\overset{\simeq}{\longrightarrow} D.$ 
						
						Then for any finite set $J,$ one has the levelwise weak equivalence $  \lambda( C(J))\overset{\simeq}{\longrightarrow} D(J).$
						Since $C(J)$ is cofibrant and $D(J)$ fibrant, its adjoint $ C(J)\overset{\simeq}{\longrightarrow} R( D(J))$ is also a weak equivalence in $\mathcal{D}.$ Therefore one has the weak equivalence $ C\overset{\simeq}{\longrightarrow}\overline{R}(D).$ 
						\item Let $A$ and $B$ be  two symmetric sequences in $\mathcal{D}.$ We have
						
						$( \overline{R}(A), \overline{R}(B))(f:J\longrightarrow I)= R(A(I))\wedge\underset{i\in I}{\bigwedge} R( B(f^{-1}(i)))$ 
						
						\ \ \ \ \ \ \ \ \ \ \ \ \ \ \ \ \ \ \ \ \ \ \ \ \ \ \ \ \ \ \ \ \ \ \	$\overset{\varphi_{A,B}^{R}}{\longrightarrow}  R(A(I)\wedge\underset{i\in I}{\bigwedge} B(f^{-1}(i)))=\overline{R}(A,B)(f:J\longrightarrow I)$
						
						where $\varphi_{A,B}^{R}$ is the natural iteration of the lax monoidal structure morphism associated to $R.$  
						
						The adjunction of this morphism gives 
						
						\begin{center}
							$ \overline{\lambda}(\overline{R}(A), \overline{R}(B))(f:J\longrightarrow I) 	\longrightarrow (A,B)(f:J\longrightarrow I)$
						\end{center}
						By applying $colim$ one get the map
						
						\begin{center}
							$ \overline {\lambda} \underset{f\in J/FinSet_0}{colim}(\overline{R}(A), \overline{R}(B))(f) 	\longrightarrow \underset{f\in J/FinSet_0}{colim}(A,B)(f)$
						\end{center}
						and it adjoint gives what we claim to be the lax monoidal structure morphism 
						\begin{center}
							$\varphi_{A,B}^{\overline {R}}: \overline{R}(A)\circ \overline{R}(B)	\longrightarrow \overline {R} (A\circ B) $
						\end{center}
						Consider now that $(\lambda, R)$ is a weak monoidal Quillen pair. Let $A,B$ be two cofibrant objects in $[Finset, \mathcal{C}].$
						
						The adjoint of the composition 	$ A\circ B\longrightarrow \overline{R}\overline{\lambda}A\circ \overline{R}\overline{\lambda}B\overset{\varphi_{A,B}^{\overline {R}}}{\longrightarrow}\overline{R}(\overline{\lambda}(A)\circ \overline{\lambda}(B))$
						is obtained, according to our above construction, from the collection of maps: for $f:J\longrightarrow I,$
						\begin{align}\label{Levelwise_WE_composition_equation}
						\overline{\lambda} (A, B)(f)\longrightarrow& \overline{\lambda}( \overline{R}\overline{\lambda}A, \overline{R}\overline{\lambda}B)(f) \longrightarrow  (\overline{\lambda}(A), \overline{\lambda}(B))(f)
						\end{align}
						and these maps are  weak equivalences by assumption. We would like to conclude that when we apply the colimit functor we obtain the weak equivalence $\overline{\lambda} (A\circ  B) \longrightarrow \overline{\lambda}(A)\circ \overline{\lambda}(B).$ 
						\begin{enumerate}
							\item [-] We assume that $ A= \underset{\rho}{\text{colim}} A<\rho>,$ with
							single cell attachment
							\begin{center}
								$\xymatrix{ K=\underset{\alpha}{\bigvee}K_\alpha\otimes \mathbb{I}(\Sigma_\alpha) \ar[r] \ar[d]& A<\rho> \ar[d]^{}\\
									K'=	\underset{\alpha}{\bigvee}K'_\alpha\otimes \mathbb{I}(\Sigma_\alpha) \ar[r]& A<\rho +1>
								}$
							\end{center}  
							When we apply the functor $-\circ B$ and the functor $\overline{\lambda},$ we obtain the following cube where the top and back faces are  the pushout diagrams
							
							\begin{center}
								\begin{tikzcd}[row sep=scriptsize, column sep=scriptsize]
									&\overline{\lambda}( K\circ B) \arrow[dl,  "\simeq"] \arrow[rr] \arrow[dd, tail]  & & \overline{\lambda}(  A<\rho>\circ B ) \arrow[dl, "\simeq "] \arrow[dd, ] \\ \overline{\lambda}( K) \circ \overline{\lambda}(B) \arrow[dd, tail, ""] \arrow[rr, crossing over]  & & \overline{\lambda}( A<\rho>) \circ \overline{\lambda}(B) \arrow[dd, ]  \\
									& \overline{\lambda}( K' \circ B) \arrow[dl, "\simeq "] \arrow[rr] & & \overline{\lambda}(  A<\rho+1>\circ B )  \arrow[dl, "h "] \\
									\overline{\lambda}( K') \circ \overline{\lambda}(B)  \arrow[rr] & & \overline{\lambda}( A<\rho +1>)  \circ \overline{\lambda}(B) \\
								\end{tikzcd}
							\end{center}
							where the horizontal  weak equivalences are induced by equation $( \ref{Levelwise_WE_composition_equation}).$ The left vertical cofibrations are induced by Lemma \ref{stable_cofibration_lemma} since $\overline{\lambda}(B)$ is cofibrant and $\overline{\lambda}$ preserves cofibrations. The front and back faces are then homotopy pushout squares, and one deduces that $h$ is a weak equivalence.
							\item[-] If $A$ is any arbitrary cofibrant object, then it is a retract of a cell object $X$, and one
							deduce easily that  $\overline{\lambda} (A\circ  B) \longrightarrow \overline{\lambda}(A)\circ \overline{\lambda}(B)$  is as weak equivalence as the retract of the weak equivalence $\overline{\lambda} (X\circ  B) \longrightarrow \overline{\lambda}(X)\circ \overline{\lambda}(B).$ 
						\end{enumerate}	
					\end{enumerate}

						\end{Proof of proposition Lif adjunction symmetric sequences}
				
				\section{Simplicial operads}
				\subsection{The semi-model category of operads}
				
				Let $(\mathcal{C}, \wedge, \mathbb{I}_\mathcal{C})$ be a cofibrantly generated symmetric monoidal model category  with cofibrant unit $\mathbb{I}_\mathcal{C}.$ 	We denote by $Op_\mathcal{C}$ the category of connected operads over $\mathcal{C}.$ These are operads $P$ so that $P(0)=\emptyset$ (the initial object) and $P(1)=\mathbb{I}_\mathcal{C}.$ 
				
				There are the following adjoint functors: \begin{center}
					$ F : [FinSet, \mathcal{C}] \rightleftarrows Op_\mathcal{C}: U $
				\end{center}
				where $F$ denotes the free operad functor, and $U$ is the forgetful functor. 
				
				To define a model category structure on $Op_\mathcal{C},$ the usual method is to transfer the model structure from $[FinSet, \mathcal{C}]$ to the category $Op_\mathcal{C}$ using the Kan's Lifting Theorem (see \cite[11.3.2]{Hir03}). Unfortunately, it is often the case that the hypothesis of the Kan theorem are not met and we get on $Op_\mathcal{C}$ a so called "semi-model structure" which is a model structure with weaker requirements for the lifting and the factorization axioms.

				\begin{definition}[semi-model category,  \cite{Hovey98}, \cite{Spitzweck}, \cite{Fresse2008BenoitFM}]
					The structure of a semi-model structure consists of a  category $\mathcal{M}$ equipped with classes of weak-equivalences, cofibrations and fibrations so that the axioms M1, M2, M3
 of model categories hold, but where the lifting axiom M4 and the factorization axiom M5 are replaced by the weaker requirements:
 
 	\begin{enumerate}
 		
 		\item[\text{M4'}.] The fibrations (resp. acyclic fibrations) have the right lifting property with respect to acyclic cofibrations (resp. cofibrations). 
 		\item[\text{M5'}.]
 		\begin{enumerate}
 			\item [-] Any morphism $f:A\longrightarrow B$ such that $A$ is cofibrant has a factorization $f=pi, $ where $i$ is a cofibration and $p$ is an acyclic fibration.
 			\item[-]  Any morphism $f:A\longrightarrow B$ such that $A$ is cofibrant has a factorization $f=qj, $ where $j$ is an acyclic  cofibration and $q$ is a fibration.
 		\end{enumerate}
 		
 			\item[\text{M0'}.] The initial object of $\mathcal{M}$
 		is cofibrant				
 	\end{enumerate}

				\end{definition}
				The semi-model category defined on the category $Op_\mathcal{C}$ is given by (see \cite[Thm 3]{Spitzweck}):
				\begin{enumerate}
					\item [-] Fibrations (resp. weak equivalences) are levelwise fibrations(resp. weak equivalences) in the underlying category $\mathcal{C}.$
					\item[-] Cofibrations are morphisms with $\Sigma_*$-cofibrant domain which have  the left lifting property with respect to trivial fibrations.
				\end{enumerate}
				 In the particular case that  $\mathcal{C}=Ch_{\geq t}(\Bbbk) , t\in \mathbb{N}\cup\{-\infty\} ,$ with   $\Bbbk$ be a field of characteristic $0,$ it is proved in \cite[ Thm 3.1]{Berger02axiomatichomotopy} that this semi-model structure is a strict model structure on $Op_\mathcal{C}.$
				 
				 There is the notion of Quillen adjunction in the context of semi-model categories (see \cite[$Â§$ 12.1.8, p.191]{BF09}).
				 \begin{definition}[Quillen adjunction between semi-model categories]
				 	Let $\lambda: \mathcal{C} \rightleftarrows \mathcal{D}: R$ be an adjunction between semi-model categories.
				 	\begin{enumerate}
				 		\item The pair $(\lambda, R)$ is a Quillen adjoint pair if $R$ preserves fibrations and trivial fibrations. 
				 		\item A Quillen  pair $(\lambda, R)$ is a Quillen equivalence  if we have further:
				 	\begin{enumerate}
				 		\item For every cofibrant object $X \in \mathcal{C},$ the composite 
				 		\begin{center}
				 			$ X \overset{\eta_X}{\longrightarrow} R\lambda(X) \overset{R(i)}{\longrightarrow} R(\lambda(X)^f)$
				 		\end{center}
				 		is a weak equivalence, where $\eta_X$ refers to the adjunction unit, and $i$ is a fibration resolution that arises from the factorization $\lambda(X)\overset{i}{\underset{\simeq}{\rightarrowtail}}\lambda(X)^f \twoheadrightarrow *$ of the terminal morphism $\lambda(X) \longrightarrow *.$
				 		
				 		\item For every cofibrant object $Y\in \mathcal{D},$ the composite 
				 		\begin{center}
				 			$ \lambda(R(Y)^c) \overset{\lambda(j)}{\longrightarrow}  \lambda(R(Y))  \overset{\varepsilon_Y}{\longrightarrow} Y$
				 		\end{center}
				 		is a weak equivalence, where $\varepsilon_Y$ refers to the adjunction counit, and $j$ is a fibration resolution that arises from the factorization $*  \rightarrowtail R(Y)^c \overset{j}{\underset{\simeq}{\twoheadrightarrow}} R(Y)$ of the initial morphism $*\longrightarrow R(Y).$
				 	\end{enumerate}	
				 	
				 	\end{enumerate}
				 \end{definition}

				\subsection{Cofibrant resolution of chain operads}
				We assume in this part that the field $\Bbbk$ is of characteristic $0.$
				Let $\mathcal{C}=Ch_{\geq t}(\Bbbk) , t\in \mathbb{N}\cup\{-\infty\} $ with the cofibrantly generated model structure defined in $§ 1.1.$
			
		Given a connected chain operad $P,$ since in this case $P$ is a $\Sigma_*$-cofibrant operad, then 	it is showed in $\cite[\text{Thm } 5.1.  \text{and  Thm } 8.5.4. ]{BergerMoerdjik}$ that the reduced $W$-construction $W^{red}(N_*^\Bbbk (\triangle^1) , P )\cong B^cB(P)$ associated to $P$ is its cofibrant model. In particular, there is a quasi-isomorphism
				\begin{center}
					$B^cB(P) \overset{\simeq}{\longrightarrow} P.$
				\end{center}

				\subsection{Realization of simplicial operads}
				We assume in this part that the field $\Bbbk$ is of characteristic $0.$
				We will be working here in the specific context where $\mathcal{C}= Ch_{\geq t}(\Bbbk) , t\in \mathbb{N}\cup\{-\infty\}.$ 
				
				In this section, we make a connection with the cofree construction on coalgebras of $§ \ref{Cofree_section}.$ We consider the geometric realization functors $\lvert -\lvert _\mathcal{C}$ (in the category $\mathcal{C}$)  and  $\lvert -\lvert _{op_\mathcal{C}}$ (in the associated category of connected operads). We establish the following comparison theorem.
							\begin{theorem}\label{Equivalence_between_realization_functors}
								Let $P_\bullet$ be a simplicial operad on chain complexes $\mathcal{C}=Ch_{\geq t}(\Bbbk) ,  t\in \mathbb{N}\cup\{-\infty\} .$  Then for any integer $r\geq 0, $ there is a quasi-isomorphism
								\begin{center}
									$\Gamma : \arrowvert B^c(B(P_\bullet))(\underline{r})\arrowvert_\mathcal{C} \overset{\simeq}{\longrightarrow}  \arrowvert P_\bullet\arrowvert_{ Op_\mathcal{C}}(\underline{r}).$
								\end{center}
							\end{theorem}
				
We give the proof of this result at the end of this section. In a preliminary step we make explicit the cosimplicial frame  $P\otimes \triangle^\bullet,$ associated to an operad $P,$  in order to fix a model for  $\arrowvert -\arrowvert_{Op_{\mathcal{C}}}.$ We will need the next proposition  which gives a cosimplicial resolution of $\Bbbk$ in CDGC.	
			
				\begin{proposition}\label{coFree_Cocomutative_Coalgebra}
					There exists a cosimplicial cocommutative non negatively differential graded coalgebra $C(\triangle^{\bullet})$ so that one have a diagram of coalgebras:	\begin{center}
						$\xymatrix{ \Bbbk \otimes sk_0 \triangle^\bullet \text{  } \ar@{>->}[r]^-{ \beta^\bullet} & C(\triangle^\bullet)\ar[r]_-{\simeq}^-{\eta^\bullet} & \Bbbk
						}$
					\end{center} 
					where \begin{enumerate}
						\item [-]
				$\Bbbk \otimes sk_0 \triangle^n :=\underset{i=0}{\overset{n}{\oplus}} \Bbbk g_i$ is the cocommutative coalgebra with the coproduct $\triangle g_i:=g_i\otimes g_i$ and $deg(g_i)=0.$
					\item[-] The second map $\eta^\bullet $ is a levelwise quasi-isomorphism. Namely $,\forall n,  \eta^n: C(\triangle^n)\overset{\simeq}{\longrightarrow} \Bbbk $ 
						\end{enumerate}
				\end{proposition}

	Now let $P$ be an operad.  Since the tensor product of a cocommutative coalgebra with a cooperad is again a cooperad, we deduce from Proposition \ref{coFree_Cocomutative_Coalgebra}  the following diagram in the category $coOp_\mathcal{C}$ of cooperads on $\mathcal{C}=Ch_{\geq t}(\Bbbk):$  $\forall k\geq 0$
	\begin{center}
		$\xymatrix{ B(P) \overline{\otimes} sk_0 \triangle^k \text{  } \ar@{>->}[r]^-{ \beta_k} & B(P) \overline{\otimes} C(\triangle^k) \ar[r]^-{\simeq}& B(P)
		}$
	\end{center}
	where, 
	\begin{enumerate}
		\item [-] $B(P)$ denotes the bar construction of $P;$
		\item[-] $B(P) \overline{\otimes} C(\triangle^k)(1):= \Bbbk;$
		\item[-] $ \forall r\geq 2, B(P) \overline{\otimes} C(\triangle^k)(r):= B(P)(r) \otimes C(\triangle^k).$
	\end{enumerate}
	By applying the cobar construction functor $B^c (-)$ to this sequence one gets the operad diagram:
	\begin{center}
		$\xymatrix{ B^c(B(P) \overline{\otimes} sk_0 \triangle^\bullet) \text{  } \ar@{>->}[r]^-{ \beta_\bullet} & B^c(B(P)\overline{\otimes} C(\triangle^\bullet)) \ar[r]^-{\simeq}& B^c(B(P)) \ar[r]^-{\simeq} & P
		}$
	\end{center}
	One deduces without too much effort from this on  the later sequence that \begin{center}
		$P\otimes \triangle^{\bullet}:=B^c(B(P)\overline{\otimes} C(\triangle^\bullet))$ 
	\end{center} 
	satisfies all the hypothesis of cosimplicial frames (see \cite[$Â§$3.2.2]{BF17}) associated to $B^cB(P).$

				\begin{definition}[Realization]\label{Realization functor}
					If $P_\bullet$ is a simplicial operad on chain complexes $\mathcal{C}=Ch_{\geq t}(\Bbbk) , t\in \mathbb{N}\cup\{-\infty\} ,$ then the realization $|P_\bullet|_{Op_\mathcal{C}} $ of $P_\bullet$ is given by the coend in $Op_\mathcal{C}$
					\begin{center}
						$|P_\bullet|_{Op_\mathcal{C}}:= \int ^{\underline{k}\in \triangle} P_k \otimes \triangle^k$
					\end{center}
				\end{definition}	
				
	We now give the proof of Proposition \ref{coFree_Cocomutative_Coalgebra} which is the combination of the next two lemma.

				\begin{mylemma}\label{Lemma 1 coFree Cocomutative }
						There exists a cosimplicial cocommutative non negatively differential graded coalgebra $C(\triangle^{\bullet})$ so that one has a diagram of coalgebras:	\begin{center}
							$\xymatrix{ \Bbbk \otimes sk_0 \triangle^\bullet \text{  } \ar@{>->}[r]^-{ \beta^\bullet} & C(\triangle^\bullet)
							}$
						\end{center} 
						where 
							$\Bbbk \otimes sk_0 \triangle^n :=\underset{i=0}{\overset{n}{\oplus}} \Bbbk g_i$ is the cocommutative coalgebra with the coproduct $\triangle g_i:=g_i\otimes g_i$ and $deg(g_i)=0.$
					
				\end{mylemma}

				\begin{proof}
				\begin{enumerate}
					\item 
					We give an inductive construction of the cosimplicial coalgebra $C(\triangle^{\bullet}).$
					We start with $C(\triangle ^0):=\Bbbk.$
					The morphism of chain complexes \begin{center}
						$C(sk_0 \triangle^1):=\Bbbk g_0\oplus \Bbbk g_1 \longrightarrow 0$
					\end{center} factors as: \begin{center}
						$\xymatrix{ C(sk_0 \triangle^1) \ar[rr]\ar@{>->}[rd]^-{\alpha^1}& & 0\\
							& V^1 \ar[ru]^{\simeq} &
						} $ 
					\end{center}	
					Since the domain of $\alpha^1$ is a cocommutative coalgebra, using the universal property of cofree cocommutative coalgebras described in Proposition \ref{Cofree Cocommutative_Coalgebra},  we co-extends $\alpha^1$ to:
				\begin{center}
					$\xymatrix{ \Bbbk g_0\oplus \Bbbk g_1 \text{ \ \  } \ar@{>->}[r]^-{\beta^1} & C(V^1)
					}$
				\end{center} 
				
		$\beta_1$ is a cofibration because of the following cocommutative diagram
		\begin{center}
			$\xymatrix{ \Bbbk g_0\oplus \Bbbk g_1 \text{ \ \  } \ar[rr]^-{\beta^1} \ar@{>->}[rrd]^-{\alpha^1}& & C(V^1) \ar@{->>}[d]^{}\\
				&& V^1
			} $ 
		\end{center}	
		 		where the vertical map is the natural projection.
				
	In addition according to Lemma \ref{C(V)_contractible_lemma}, one has $C(V^1)\overset{\eta}{\underset{\simeq}{\longrightarrow}} \Bbbk.$ Therefore we have the following diagram in the category of coalgebras:
					\begin{center}
						$\xymatrix{ \Bbbk g_0\oplus \Bbbk g_1 \text{  } \ar@{>->}[r]^-{\beta^1} & C(V^1) \ar[r]_-{\simeq}^-{\eta}& \Bbbk
						}$
					\end{center}

				 We then set: 
					\begin{enumerate}
						\item [-] $C(cosk_0 \triangle^1):=C(\triangle ^0)= \Bbbk$
						\item [-] $C(\triangle^1):=C(V ^1)\otimes C(cosk_0 \triangle^1)\cong C(V ^1);$
						The cofaces $d^0,d^1: C(\triangle^0) \longrightarrow C(\triangle^1)$ are given by the respective restrictions of $\beta^1$ on $\Bbbk g_0$ and on $\Bbbk g_1.$ The codegeneracy  $s^0: C(\triangle^1) \longrightarrow C(\triangle^0)$   is given by $ \eta.$

					\end{enumerate}
					If $k=2,$ we first  use the notation $C(sk_1 \triangle^2)$ to be the colimit, as a coalgebra,  of the diagram 
					\begin{center}
						$ \xymatrix{ & &C(\triangle^0_{(2)}) \ar[ld] \ar[dr] &\\
							&	C(\triangle^1_{(02)}) & &C(\triangle^1_{(12)})\\
							C(\triangle^0_{(0)})\ar[ur]\ar[rr]& & C(\triangle^1_{(01)}) &&C(\triangle^0_{(1)})\ar[ll] \ar[lu]   	  
						}$
					\end{center}	
					where for any $i,j \in \{0,1,2\}, $  $C(\triangle^0_{(i)})= \Bbbk g_i$ and refers to the vertices of the standard simplex $\triangle^2;$ $C(\triangle^1_{(ij)})$ are copies of $C(\triangle^1)$ whose indices refer to the 1-simplices of the standard simplex $\triangle^2.$
					As in the previous case, we consider the following factorization of the trivial map:
					\begin{center}
						$\xymatrix{ C(sk_1 \triangle^2) \text{  } \ar[rr]\ar@{>->}[rd]^-{\alpha^2}& & 0\\
							& V^2 \ar[ru]^{\simeq} &
						} $ 
					\end{center}	
					and using the same previous trick, one recovers from $\alpha_2$ a coalgebra morphism 
					\begin{center}
						$\xymatrix{ C(sk_1 \triangle^2) \text{  } \ar@{>->}[r]^-{ } & C(V^2) \ar[r]^-{\simeq}& \Bbbk
						}$
					\end{center}  
					In addition one has an inclusion $\Bbbk \otimes sk_0 \triangle^2 \rightarrowtail C(sk_1 \triangle^2),$ thus one form the morphism 
					\begin{center}
						$\xymatrix{ \Bbbk \otimes sk_0 \triangle^2 \text{  } \ar@{>->}[r]^-{ \beta^2} & C(V^2) \ar[r]^-{\simeq}& \Bbbk
						}$
					\end{center}
					We then set 
					\begin{enumerate}
						\item [-] $C(cosk_1 \triangle^2):=lim (\xymatrix{ C(\triangle^1)\ar[r]^-{s^0} & C(\triangle^0)& C(\triangle^1)\ar[l]_-{s^0} })$
						\item[-] $C(\triangle^2):= C(V^2)\otimes C(cosk_1 \triangle^2);$ The cofaces $d^0, d^1, d^2: C(\triangle^1)\longrightarrow C(\triangle^2)$ are given , with our construction,  by the respective restriction of $\alpha_2$ on the cofaces $C(\triangle^1_{(ij)}).$  The codegeneracies $s^0, s^1: C(\triangle^2) \longrightarrow C(\triangle^1)$ are obtained from the projection $C(V^2)\otimes C(cosk_1 \triangle^2) \longrightarrow  C(cosk_1 \triangle^2)$ followed respectively by the projections $C(cosk_1 \triangle^2) \longrightarrow C(\triangle^1).$
					\end{enumerate} 
					This construction generalizes inductively to higher values according to the shape of the standard simplexes $\triangle^n  (n\geq0), $ as follows:
					\begin{enumerate}
						\item[-] $C(sk_{n-1}\triangle^{n}):=\underset{n>\underline{m} \rightarrow \underline{n}}{colim} C(\triangle^m);$ One forms the factorization 
						\begin{center}
							$\xymatrix{ C(sk_{n-1} \triangle^n) \text{  } \ar@{>->}[r]^-{ } & C(V^n) \ar[r]^-{\simeq}& \Bbbk
							}$
						\end{center}
						and then deduce $\beta^n: \Bbbk \otimes sk_0 \triangle^n \rightarrowtail C(V^n) $ by precomposing the above injection  with the sequence 
						$\Bbbk \otimes sk_0 \triangle^n\rightarrowtail C(sk_1 \triangle^n) \rightarrowtail ...\rightarrowtail C(sk_{n-1} \triangle^n).$
						\item [-] $C(cosk_{n-1} \triangle^n):=\underset{\underline{n} \rightarrow \underline{m}< n}{lim} C(\triangle^m)$
						\item[-] $C(\triangle^n):= C(V^n)\otimes C(cosk_{n-1} \triangle^n);$ The cofaces $d^{i}: C(\triangle^{n-1})\longrightarrow C(\triangle^{n})$ and the codegeneracies $s^j: C(\triangle^n) \longrightarrow C(\triangle^{n-1})$ are obtained,  as in the lower cases, by construction.
					\end{enumerate}
			
				\item	Finally We need to show that $\beta_\bullet$ is a Reedy cofibration.
				In fact, one can see that $\forall r\geq0,$ the $r-$th latching object $L^r \Bbbk\otimes sk_0 \triangle^\bullet$ of the cosimplicial object $\Bbbk\otimes sk_0 \triangle^\bullet$ is 
				\begin{center}
					$L^r \Bbbk\otimes sk_0 \triangle^\bullet \cong  \Bbbk\otimes sk_0 \triangle^r$
				\end{center}
				Therefore proving that $\beta_\bullet$ is a Reedy cofibration reduces to proving that $\forall r,$ the map $L^rC(\triangle^\bullet)=C(sk_{r-1}\triangle^r) \longrightarrow C(\triangle^r)$ is a cofibration in $\text{co-}\mathcal{CA}$ and this is done through the construction of $C(\triangle^\bullet).$

						\end{enumerate}
				\end{proof}

				\begin{mylemma}\label{C delta is acyclic lemma}
					Using the notations of Proposition \ref{coFree_Cocomutative_Coalgebra} there is $,\forall n\geq 0,$ an isomorphism

						\begin{center}
							$C(\triangle^n)\cong \underset{ \underset{n\geq k}{v: \underline{n} \twoheadrightarrow \underline{k}}}{\bigotimes} C(V^k)_v$
						\end{center}
					
					where given any surjection $v: \underline{n} \twoheadrightarrow \underline{k},$  $C(V^k)_v$ is the coalgebra $ C(V^k)$ labeled by $v.$
				\end{mylemma}
				
				\begin{proof}
					We give the proof by induction on $n.$ When $n=0,$ then $C(\triangle^0)=\Bbbk$ satisfies the hypothesis. We now consider that we have constructed ,for a given $n,$ the isomorphisms  
					\begin{center}
							$C(\triangle^r)\cong \underset{\underset{n\geq k}{v: \underline{n} \twoheadrightarrow \underline{k}}}{\otimes} C(V^k)_v,$ for $r\in \{0, ..., n\}.$
					\end{center}

					Since $C(\triangle^{n+1})=C(V^{n+1})\otimes C(cosk_n\triangle^{n+1}),$  it remains in proving that 		
						\begin{center}
							$C(cosk_n\triangle^{n+1})\cong   \underset{	\underset{n+1\geq k}{v: \underline{n+1} \twoheadrightarrow \underline{k}}}{\otimes} C(V^k)_v.$
						\end{center}

					We consider the morphism 
					\begin{center}
						$ \Phi: C(cosk_n\triangle^{n+1})=\underset{\underset{n+1 >m}{u: \underline{n+1}\twoheadrightarrow \underline{m}}}{\text{ lim}} C(\triangle^m)_u \longrightarrow    \underset{\underset{n+1>k}{ v: \underline{n+1} \twoheadrightarrow \underline{k}}}{\otimes} C(V^k)_v$
					\end{center}

					given by the morphisms: for any surjection $u:\underline{n+1}\twoheadrightarrow \underline{m},$ 
					\begin{center}
						$\Phi_u: C(\triangle^m)_u= \underset{f: \underline{m}\twoheadrightarrow \underline{k}}{\otimes} C(V^k)_f \longrightarrow \underset{\underset{n+1 >k}{ v: \underline{n+1} \twoheadrightarrow \underline{k}}}{\otimes} C(V^k)_v$
					\end{center}
					is the map which is identity on the terms
					\begin{center}
					$	C(V^k)_f \overset{=}{\longrightarrow } C(V^k)_v$ when $v=f\circ u$
					\end{center}
				On the other hand, 	We define the morphism 
					\begin{center}
						$ \Psi: \underset{ \underset{n+1>k}{v: \underline{n+1} \twoheadrightarrow \underline{k}}}{\otimes} C(V^k)_v \longrightarrow C(cosk_n\triangle^{n+1})=\underset{ \underset{n+1>m}{u: \underline{n+1}\twoheadrightarrow \underline{m}}}{\text{ lim}} C(\triangle^m)_u $
					\end{center} 
					given by the morphisms: for any surjection $u:\underline{n+1}\twoheadrightarrow \underline{m},$ 
					\begin{center}
						$\Psi_u:  \underset{ \underset{n+1>k}{v: \underline{n+1} \twoheadrightarrow \underline{k}}}{\otimes} C(V^k)_v \longrightarrow C(\triangle^m)_u= \underset{f: \underline{m}\twoheadrightarrow \underline{k}}{\otimes} C(V^k)_f  $
					\end{center}
					is the map which is identity on the terms
					\begin{center}
						$	C(V^k)_v \overset{=}{\longrightarrow } C(V^k)_f$ with $v=f\circ u$
					\end{center}
					$\Psi_u$ is well defined since given an integer $k$ such that $k\leq m $ and a surjection $v: \underline{n+1}\twoheadrightarrow \underline{k},$ there is a unique surjection $f:\underline{m} \twoheadrightarrow \underline{k}$ so that $v=f\circ u.$
					
					One can see by construction that $\Psi_u \Phi_u=Id$ and that $\Phi_u \Psi_u=Id.$ This proves that   $\Psi \Phi=Id$ and that $\Phi \Psi=Id.$
				\end{proof}

		\begin{Proof of Proposition CoFree Cosimplicial  Cocommutative Coalgebra}
			Lemma \ref{Lemma 1 coFree Cocomutative } gives completly the first part of the proposition.

			 On the other hand, the proof that $C(\triangle^n) \simeq \Bbbk$ follows from Lemma \ref{C delta is acyclic lemma} (where we make a more explicit description of   $C(\triangle^n)$ ) and the fact that  $\forall m,  C(V^m)\simeq \Bbbk$ which follows by construction of $C(V^\bullet)$ in the proof of  	Lemma \ref{Lemma 1 coFree Cocomutative } .
			
		\end{Proof of Proposition CoFree Cosimplicial  Cocommutative Coalgebra}

			We end this section by giving the proof of Theorem \ref{Equivalence_between_realization_functors} which uses a spectral sequence argument.  For this, we start by defining a filtration for the cobar construction associated to a cooperad.
				
				\begin{notation}
					Let $T$ be a tree. If $M$ is any symmetric sequence in chain complexes, it will be more simple to adopt the notation $T(M):= \underset{v\in V(T)}{\otimes}M(I_v),$ where $V(T)$ denotes the set of the vertices of $T$ and given a vertex $v,$ $I_v$ is the set of all incoming edges to $v.$
					
				\end{notation}
				
				\begin{definition}[Filtering the cobar construction]\label{filtering_cobar_construction}
					Let $Q$ be a chain cooperad. We define the filtration $	\mathcal{F}_{ u}B^c(Q)$ of the cobar construction of the form:
					\begin{center}
						$ B^c(Q)=\mathcal{F}_{0} B^c(Q) \supseteq	\mathcal{F}_{ 1} B^c(Q)\supseteq ... \supseteq	\mathcal{F}_{  u}  B^c(Q)\supseteq ... $
					\end{center}
					as follows:  \begin{center}
						$	\mathcal{F}_{ u}B^c(Q):= \underset{T ,  |V(T)|\geq   u}{\bigoplus} T(\Sigma^{-1}\overline{Q}) , $
					\end{center}
					where $|V(T)|$ denotes the number of internal vertices of the tree $T.$
					
					The cobar differential $\partial_{cobar},$ which is defined from the cooperad coproduct $Q\longrightarrow Q\circ Q,$ increases the number of vertices of the tree. Therefore $\partial_{cobar} 	\mathcal{F}_{ u} B^c(Q) \subseteq \mathcal{F}_{u+1} B^c(Q).$
				\end{definition}
				The good characteristic of this filtration is that it isolates the  cobar differential $\partial_{cobar}$ from the first page of the associate spectral sequence. Namely, the associated bigraded complex  $\mathcal{F}_{*,*}B^c(Q)$ is given by
				\begin{center}
					$ \mathcal{F}_{*,*}B^c(Q) \cong F(\Sigma^{-1}\overline{Q}).$
				\end{center}

				\begin{Proof of Theorem Equivalence between realization functors}

					one can observe that: 
					\begin{align*}
					\arrowvert P_\bullet\arrowvert_{ Op_\mathcal{C}}(r) \overset{}{\cong}& B^c(\int_{coOp_\mathcal{C}}^{\underline{k}\in \triangle} B(P_k)\otimes C(\triangle^k))(r)  \text{ ( This is because $B^c(-)$ is a left adjoint functor )}\\
					\overset{(1)}{\cong}& B^c(\int_{[Finset, \mathcal{C}]}^{\underline{k}\in \triangle} B(P_k)\otimes C(\triangle^k))(r) , 
					\end{align*} 
					where:
					\begin{enumerate}
						\item[-] $coOp_\mathcal{C}$ denotes the category of cooperads on $\mathcal{C};$
						\item[-] The isomorphism $(1)$ is justified by the fact that the forgetful functor $U: coOp_{\mathcal{C}} \longrightarrow [Finset, \mathcal{C}]$ preserves colimits as a left adjoint. 
					\end{enumerate}

					On the other hand the  model of the realization functor in chain complexes $\arrowvert B^c(B(P_\bullet))(r)\arrowvert_\mathcal{C}$ that we use in this proof is the coend 
					\begin{center}
						$	\arrowvert B^c(B(P_\bullet))(r)\arrowvert_\mathcal{C} \simeq \int^{\underline{k}\in \triangle}_\mathcal{C} B^c(B(P_k))(r)\otimes C(\triangle^k)$
					\end{center}
					After this setting we want to construct a chain map
					\begin{center}
						$\Gamma: \int^{\underline{k}\in \triangle}_\mathcal{C} B^c(B(P_k))(r)\otimes C(\triangle^k) \longrightarrow B^c(\int_{[Finset, \mathcal{C}]}^{\underline{k}\in \triangle} B(P_k)\otimes C(\triangle^k))(r) $
					\end{center}
					The construction is made on the set of trees. Namely let $T$ be a tree with $r$ leaves. We fix an order on the set of the vertices of the tree $T$ and 
					we define the chain map $\Gamma_T$ as follows:
					\begin{align*}
					\Gamma_T: \int^{\underline{k}\in \triangle}_\mathcal{C} T(s^{-1}B(P_k))\otimes C(\triangle^k) & \longrightarrow T(s^{-1}\int_{[Finset, \mathcal{C}]}^{\underline{k}\in \triangle} B(P_k)\otimes C(\triangle^k)) \\
					T\otimes c=\{\gamma_v, v\in V(T) \}\otimes c &\longmapsto  \{\gamma_v \otimes c_v, v\in V(T) \}
					\end{align*}
					where \begin{align*}
					\triangle: C(\triangle^k)&\longrightarrow C(\triangle^k)^{\otimes V(T)}\\
					v &\longmapsto  \underset{v \in V(T)}{\otimes} c_v  \text{    \ \  \ \ \ \ \ \  \ \ \ \ \ \ is the coproduct of the coalgebra $C(\triangle^k).$}
					\end{align*}

					The collection of the maps $\Gamma_T$ produces naturaly the chain map $\Gamma$ that we want.
					
					We now have to prove that $\Gamma$ is a quasi-isomorphism. To prove it we use a spectral sequence argument from the filtration of the cobar construction of Definition \ref{filtering_cobar_construction}. Namely, let us build the filtrations

					\begin{align*}
					\mathcal{F}_{ u} \arrowvert B^c(B(P_\bullet))(r)|_\mathcal{C}:=& \int^{\underline{k}\in \triangle}_\mathcal{C}  \mathcal{F}_uB^c(B(P_k) ) \otimes C(\triangle^k)\\ 
					&\text{and } \\
					\mathcal{G}_{ u} \arrowvert P_\bullet\arrowvert_{ Op_\mathcal{C}}:= &\mathcal{F}_uB^c(\int_{[Finset, \mathcal{C}]}^{\underline{k}\in \triangle} B(P_k)\otimes C(\triangle^k) )
					\end{align*}
					
					One can see that the associated bigraded complexes   $\mathcal{F}_{*, *} \arrowvert B^c(B(P_\bullet))(r)|_\mathcal{C}$
					and $	\mathcal{G}_{*,*} \arrowvert P_\bullet\arrowvert_{ Op_\mathcal{C}}$ are equivalent to:
					\begin{align*}
					\mathcal{F}_{*, *} \arrowvert B^c(B(P_\bullet))(r)|_\mathcal{C} \cong& \int^{\underline{k}\in \triangle}_\mathcal{C}  F(\Sigma^{-1}\overline{B}(P_k) ) \otimes C(\triangle^k)
					\end{align*}
					and 
					\begin{align*}
					\mathcal{G}_{u,t} \arrowvert P_\bullet\arrowvert_{ Op_\mathcal{C}}=& F(\Sigma^{-1} \int_{[Finset, \mathcal{C}]}^{\underline{k}\in \triangle} B(P_k)\otimes C(\triangle^k) )
					\end{align*}
					
					From a classical theorem of spectral sequences, to prove that $\Gamma$ is a quasi-isomorphism, it will be enough to show that its restriction to the bigraded complexes 
					$\mathcal{F}_{*, *} \arrowvert B^c(B(P_\bullet))(r)|_\mathcal{C}$
					and $	\mathcal{G}_{*,*} \arrowvert P_\bullet\arrowvert_{ Op_\mathcal{C}}$  is a quasi-isomorphism with the internal differential of these complexes. This later condition is true if for any fixed tree $T,$ the associated chain map $\Gamma_T$ is a quasi-isomorphism with the internal differential.

					Let us consider now the following commutative diagram
					\begin{center}
						$\xymatrix{ \int^{\underline{k}\in \triangle}_\mathcal{C} T(s^{-1}B(P_k))\otimes C(\triangle^k) \ar[rr]^{\Gamma_T}&&T(s^{-1}\int_{[Finset, \mathcal{C}]}^{\underline{k}\in \triangle} B(P_k)\otimes C(\triangle^k)) \\ 
							\int^{\underline{k}\in \triangle}_\mathcal{C} T(s^{-1}B(P_k))\otimes C(\triangle^k)\otimes N_*(\triangle^{k}) \ar[rr]^-{\Gamma_T \otimes AW_*} \ar[u]^{\simeq} \ar[d]_{\simeq}&&T(s^{-1}\int_{[Finset, \mathcal{C}]}^{\underline{k}\in \triangle} B(P_k)\otimes C(\triangle^k) \otimes N_*(\triangle^k))  \ar[u]_{\simeq} \ar[d]^{\simeq} \\
							\int^{\underline{k}\in \triangle}_\mathcal{C} T(s^{-1}B(P_k))\otimes N_*(\triangle^{k})  \ar[rr]^-{\widetilde{AW_*}}\ar@{}[d]|*[@]{\cong}_{\alpha}&&T(s^{-1}\int_{[Finset, \mathcal{C}]}^{\underline{k}\in \triangle} B(P_k)\otimes N_*(\triangle^k))\ar@{}[d]|*[@]{=}^-{T(s^{-1}\alpha)}\\
							N_*T(s^{-1}B(P_\bullet)) \ar[rr]^{AW_*}_{\simeq} && T(s^{-1}N_*B(P_\bullet))
						}$
					\end{center}
					where, \begin{enumerate}
						\item Given any simplicial chain complex $K_\bullet,$  $\alpha: \int^{\underline{n}\in \triangle}_\mathcal{C} K_n\otimes N_*(\triangle^{n})  \longrightarrow N_*K_\bullet$  is the chain complex defined by: $\alpha( [x\otimes \sigma ] ):= [  \sigma^{*} x],$
						where $ \sigma ^{*}: K_{n}\longrightarrow K_{*}$ is induced by $\sigma: \underline{*} \longrightarrow \underline{n}.$ 
						The inverse of $\alpha$ is the morphism $\alpha': [x]\longmapsto [x\otimes \iota],$ where $\iota$ denotes the top cell of $\triangle_n.$
						\item $\widetilde{AW_*}$  is the iteration (relative to the number of the vertices of the tree $T$) of the  following map: 
						given any two simplicial chain complexes $K_\bullet$ and $L_\bullet,$ one have 
						\begin{align*}
						\widetilde{AW}:  \int^{\underline{n} \in \triangle} K_n \otimes L_n \otimes N_*\triangle ^{n}& \longrightarrow \underset{k+l=n}{\oplus} \int^{\underline{k} \in \triangle} K_k \otimes N_*\triangle^{k}  \otimes \int^{\underline{l} \in \triangle} L_l \otimes N_*\triangle^{l} \\
						[x\otimes y\otimes \iota_n] & \longmapsto [ \iota(0 ...k)^{*}x\otimes \iota_k] \otimes [\iota(k ...n)^{*}y\otimes \iota_l] 
						\end{align*}
						where  $\iota_j$ is the top cell of $\triangle^{j};$ $\iota_n(0...k): \underline{k} \longrightarrow \underline{n}$ and $  \iota_n(k ...n): \underline{l} \longrightarrow \underline{n}$ are the canonical monotone injections  defined by: $\iota_n(0...k)(i)=i, \iota_n(k...n)(i)=k+i.$

						\item $AW_*: N_*T(s^{-1}B(P_\bullet))\overset{\simeq}{\longrightarrow} T(s^{-1}N_*B(P_\bullet))$ is the iteration (relative to the number of the vertices of the tree $T$) of the  Alexander Whitney  map $AW$ defined as follows: 
						given any two simplicial chain complexes $K_\bullet$ and $L_\bullet,$  $a\in K_n,$ and $b\in L_n,$ one have
					\begin{align*}
					AW_*:  N_*(K_\bullet \otimes L_\bullet)& \longrightarrow N_*(K_\bullet)\otimes N_*(L_\bullet)  \\
					AW(a\otimes b):= &\underset{p+q=n}{\Sigma} d^p(a)\otimes d^q_0(b)
					\end{align*}
					 where 
					the "front face" $d^p: K_n \longrightarrow K_p$ and the "back face" $d^q_0: L_n \longrightarrow L_q$ are induced by the injective monotone maps $\delta^p: \underline{p}\longrightarrow \underline{p+q}$ and $\delta^q_0: \underline{q}\longrightarrow \underline{p+q}$ defined by:
					$\delta^p(i):=i$ and $\delta^q_0(i):=p+i.$
				
						This is the same map though defined with a different notation in \cite[Chap VIII, Corollary  8.6.]{maclane}, and is known to  be a quasi-isomorphism (see \cite[$Â§$ 2.3.]{SS03}).
						
						\item The vertical weak equivalences in each column are due to the fact that the terms in the integral are all good model for simplicial framing.
					\end{enumerate}
					
					From this diagram, since the morphism  $AW_*$ at the bottom is a quasi-isomorphism, it follows inductively that  $\Gamma_T$ is a quasi-isomorphism.
					
				\end{Proof of Theorem Equivalence between realization functors}

				\subsection{Simplicial resolution of operads}
					Let $(\mathcal{C}, \wedge, \mathbb{I}_\mathcal{C})$ be a cofibrantly generated symmetric monoidal model category  equipped with colimits with cofibrant unit $\mathbb{I}_\mathcal{C}.$ In this section we associate to any operad $P$ a simplicial operad $Res_\bullet(P) $ so that the realization functor ( see Definition \ref{Realization functor} )
	applied to this object gives an operad equivalent to $P.$ The section is very short, we first make the construction of $Res_\bullet(P) $ and state its relation with $P$ in  the only proposition of the section.

				We consider the following adjoint functors: \begin{center}
					$ F : [Finset, \mathcal{C}] \rightleftarrows Op_\mathcal{C}: U $
				\end{center}
				where $F$ denotes the free operad functor and $U$ is the forgethful functor. This adjunction gives the comonad $T:=FU: Op_\mathcal{C} \longrightarrow Op_\mathcal{C}.$
				
				Let $P$ be an operad in $\mathcal{C}.$ We define the  simplicial operad $Res_\bullet(P)$ associated to the comonad $T$ and the operad $P$ as follows:
				\begin{center}
					For any integer $k,$  $Res_k(P):= T^{k+1}(P)$
				\end{center}
				The counit $\varepsilon: T\longrightarrow 1$ of the comonad $T$ is used in the classical way to construct the faces  $d_i: Res_{k}(P) \longrightarrow Res_{k-1}(P), 0 \leq i\leq k.$ The degeneracies $s_j: Res_k(P) \longrightarrow Res_{k+1}(P), 0 \leq j\leq k, $ are induced by the comonadic coproduct $\mu: T\longrightarrow T^2.$ 
				
				\begin{remark}\label{Augmentation_Simplicial_operad}
					The simplicial operad $Res_\bullet(P)$ has a natural augmentation $Res_0(P)= T(P)\overset{\varepsilon}{\longrightarrow} P.$
					
					The associated augmented simplicial  chain complex sequence $\varepsilon: URes_\bullet(P) \overset{ }{\longrightarrow} UP$ has extra degeneracies $s_{-1}: URes_k (P)\longrightarrow URes_{k+1} (P),$ $(\forall k \geq -1)$ given by:
					\begin{center}
						$\xymatrix{ UP \ar[rr]^-{s_{-1}=\eta_{UP}} & &  UFU(P)\ar[rrr]^-{s_{-1}=\eta_{URes_0 P}}\ar@{}[d]|*[@]{=}& & &   UFUFU(P)\ar[rr]^{s_{-1}}\ar@{}[d]|*[@]{=} & & ...\\
							& &  URes_0(P) & & &  URes_1 (P)
						} $
					\end{center}
					
					A straight consequence of these extra degeneracies when $\mathcal{C}=Ch_{\geq t}(\Bbbk)$ is that one have the quasi-isomorphisms : $\forall r\geq 0, $ 
					\begin{center}
						$ \varepsilon: \arrowvert Res_\bullet (P)(r)\arrowvert _{\mathcal{C}} \overset{\simeq}{\longrightarrow } P(r)$
					\end{center}
				\end{remark}

				\begin{proposition}\label{Simplicial_resolution_Operads}
					Let us consider the category $\mathcal{C}=(Ch_{\geq t}(\Bbbk), \otimes, \Bbbk).$ If $P$ is a cofibrant operad on $\mathcal{C},$ then the augmentation   $\varepsilon: \arrowvert Res_\bullet (P)\arrowvert_{Op_\mathcal{C}} \overset{}{\longrightarrow }P$ is a weak equivalence.
				\end{proposition}
				\begin{proof}
					Let us consider the following commutative diagram: $\forall r\geq 0,$
					\begin{center}
						$\xymatrix{ \arrowvert B^c B(Res_\bullet (P))(r)\arrowvert_{\mathcal{C}} \ar[rrr]^{\simeq}_-{(1)} \ar[rd]^{\simeq}_-{(2)}  & & & \arrowvert Res_\bullet (P)\arrowvert_{Op_\mathcal{C}}(r)\ar[ldd]^{(4)}\\
							& \arrowvert Res_\bullet (P)(r)\arrowvert_{\mathcal{C}}\ar[rd]^-{(3)}_{\simeq}\\
							&& P(r)
						}$
					\end{center}
					where 
					\begin{enumerate}
						\item [-] The quasi-isomorphism $(1)$ is induced by Theorem \ref{Equivalence_between_realization_functors};
						\item[-] The quasi-isomorphism $(2)$ is induced by the following fact:
						
						$P$ is cofibrant, therefore $Res_\bullet (P)$ and $B^cB(Res_\bullet (P) )$ are Reedy cofibrant. We conclude using \cite[Thm 3.3.7.]{BF17} that $(2)$ is a quasi-isomorphism.
						\item[-] $(3)$ is induced by Remark \ref{Augmentation_Simplicial_operad}.
					\end{enumerate}
					We then deduce by the $2$-out of $3$ property of weak equivalence that $(4)$ is a quasi-isomorphism.
					
				\end{proof}

				\section{Equivalence of operads on Quillen equivalent categories}
				This section can be seen as an application of the whole paper since we combine  here all the previous results. Before we state the main theorem, we  need 
to remind the construction of the left adjoint functor, that arises from a weak monoidal Quillen pair, at the level of operads. 

				Let $\lambda: \mathcal{C}=Ch_{\geq t}(\Bbbk) \rightleftarrows \mathcal{D}: R$ be a Quillen pair between the category of chain complexes $Ch_{\geq t}(\Bbbk)$ $( t\in \mathbb{N}\cup\{-\infty\})$  with its classical projective model structure, and any model category $\mathcal{D}.$
				 If in addition the category $\mathcal{D}$ is monoidal and the pair $(\lambda, R)$ is a weak symmetric monoidal Quillen pair, then the functor $R$ extends naturally to a functor $\overline{R}: Op_\mathcal{D}\longrightarrow Op_\mathcal{C}$ given by: 
				\begin{center}
					$\forall k$ and $Q\in Op_\mathcal{D},$ $ \overline{R}(Q)(k):= R(Q(k)).$
				\end{center}
				It is proved in \cite[Prop 3.1.5.-(a) ]{fresse2017homotopy} that the functor $\overline{R}$ has a left adjoint 	$L: Op_\mathcal{C}\longrightarrow Op_\mathcal{D}$
				given by:
				\begin{enumerate}
					\item [(a)] If $P=F(M)$ is a free operad in chain complexes generated by a symmetric sequence $M,$ then
					\begin{center}
						$L(F(M)):=F(\overline{\lambda}(M)),$
					\end{center} 
					where $\overline{\lambda}: [Finset,\mathcal{C} ]\longrightarrow [Finset,\mathcal{D} ]$ is the aritywise left composition with $\lambda;$
					\item[(b)] If $P$ is any operad on chain complexes, then one make the following identification: \begin{center}
						\begin{tikzcd}
							P \cong coeq(FU FU (P) \arrow[r, shift left,  yshift=0.7ex, "d_0"]
							\arrow[r, shift right,  yshift=-0.7ex, "d_1"]
							& FU (P))
						\end{tikzcd}
					\end{center}
					where \begin{enumerate}
						\item [-] $U: Op_\mathcal{C}\longrightarrow [Finset, \mathcal{C}]$ is the forgetful functor;
						\item[-] $d_0: FU FU (P)\longrightarrow FU (P)$ is the morphism of operads adjoint of the identity morphism of symmetric sequences  $Id: UFU(P)\longrightarrow UFU(P);$
						\item [-] $d_1=FU(\varepsilon): FU FU (P)\longrightarrow FU (P),$ with $\varepsilon: FU(P)\longrightarrow P$ is  the morphism of operads adjoint to the identity of symmetric sequences $Id: U(P)\longrightarrow U(P).$
					\end{enumerate}
					We now set \begin{center}
						\begin{tikzcd}
							L(P) := coeq( \underset{   =  L(FUFU(P) )}{\underbrace{F (\overline{\lambda} (UFU (P)))}} \arrow[r, shift left,  yshift=0.7ex]
							\arrow[r, shift right,  yshift=-0.7ex]
							& \underset{ =L(FU(P))  }{\underbrace{F (\overline{\lambda}(U (P)))} )}
						\end{tikzcd}
					\end{center}
				\end{enumerate}
				
				Using this notation and construction, we state the following theorem which is the main result of this section:
				\begin{theorem}\label{Quillen_equivalence_operads}
					Let $(\mathcal{D}, \wedge, \mathbb{I}_\mathcal{D})$ be a cofibrantly generated symmetric monoidal model category with cofibrant unit $\mathbb{I}_\mathcal{D},$ and let  
					$\lambda: \mathcal{C}=Ch_{\geq t}(\Bbbk)\rightleftarrows \mathcal{D}: R$ be a weak monoidal Quillen pair.  If the pair $(\lambda, R)$ is a Quillen equivalence, then the pair $L: Op_\mathcal{C} \leftrightarrows Op_\mathcal{D}: \overline{R}$ is a Quillen equivalence between the semi-model categories  $Op_\mathcal{C}$ and $Op_\mathcal{D}.$
					
					In addition, if the semi-model structure on $Op_\mathcal{D}$ is a strict model structure, then the pair $(L,\overline{R})$ is a strict Quillen equivalence.
				\end{theorem}
				
				To pove this theorem, we will need the result of the following proposition.
				\begin{proposition}\label{Relation_betwen_left_adjoint_functors}
					Let $(\mathcal{D}, \wedge, \mathbb{I}_\mathcal{D})$ be a cofibrantly generated symmetric monoidal model category, and let $\lambda: \mathcal{C}=Ch_{\geq t}(\Bbbk)\rightleftarrows \mathcal{D}: R$ be a weak monoidal Quillen pair.
					If $P$ is a cofibrant operad in $Ch_{\geq t}(\Bbbk),$ then the morphism $\overline{\lambda}(U(P))\longrightarrow UL(P)$, which is adjoint to the  unit $\eta: P \longrightarrow \overline{R}L(P),$ 
					is a weak equivalence.
				\end{proposition}

				\begin{proof}
					We form the following diagram which is commutative from the natural transformations  $|Res_\bullet (-)|_{Op_\mathcal{C}}\longrightarrow 1_{Op_\mathcal{C}}$ and $|-|_\mathcal{C} \longrightarrow U|-|_{Op_\mathcal{C}}:$  $ \forall r \geq 0,$
					\begin{center}
						$\xymatrix{ &\lambda(P(r)) \ar[r]& L(P)(r)\\
							&\lambda(|Res_\bullet (P)|_{Op_\mathcal{C}}(r))\ar[r] \ar[u]^{\simeq}_{(1)}& 	L(|Res_\bullet (P)|_{Op_\mathcal{C}})(r)\ar[u]^{\simeq}_{(2)} \ar[r]^{\cong}_{(3)}&	| L(Res_\bullet (P))|_{Op_\mathcal{C}}(r) \\
							| \lambda (Res_\bullet (P)(r))|_{\mathcal{C}} \ar[r]^{\cong}_{(5)}&	\lambda (|Res_\bullet (P)(r)|_{\mathcal{C}}) \ar[rr] \ar[u]^{(4)}_{\simeq}& & |L(Res_\bullet (P))(r)|_{\mathcal{C}}  \ar[u]^{(6)}_{\simeq}
						}
						$
					\end{center}
					where: \begin{enumerate}
						\item [-] The weak equivalence of $(1)$ is justified by the weak equivalence of Corollary \ref{Simplicial_resolution_Operads}, and the fact that the functor $\lambda$ is a left Quillen adjoint so preserves weak equivalences ( all chain complexes are cofibrant);
						\item[-] The weak equivalence of $(2)$ is justified by the following facts:  the functor $L$ is a left Quillen adjoint, thus preserves weak equivalence between cofibrant operads, and the operads $P$ an $|Res_\bullet (P)|_{Op_\mathcal{C}}$ are cofibrant , thus by applying $L$ to the weak equivalence of  Proposition \ref{Simplicial_resolution_Operads}, one obtain $(2);$
						\item[-]  The weak equivalences of $(4)$ and $(6)$ are a straight use of Theorem \ref{Equivalence_between_realization_functors};
						\item [-] The isomorphisms $(3)$ and $(5)$ come from then fact that the functors $L$ and $\lambda$  commute with colimits as left adjoint.
					\end{enumerate}	
					The horizontal map at the bottom is obtained literally by applying the functor $|-|_\mathcal{C}$ to the simplicial map
					\begin{center}
						$\lambda(( FU)^{\bullet +1}(P))(r)\longrightarrow L((FU)^{\bullet +1}(P))(r)=  F \overline{\lambda}(FU)^{\bullet +1}(P)(r)$
					\end{center}
					and this later map is a weak equivalence since it is built out of  the  lax monoidal morphisms
					$\forall V,W \in Ch^+_{\Bbbk},$ $\lambda(V\otimes W)\overset{\simeq}{\longrightarrow} \lambda(V)\wedge \lambda(W).$
					
					We therefore deduce from the bottom weak equivalence of the diagram that the middle and then the top maps are weak equivalences.
				\end{proof}
				
				\begin{Proof of main theorem}
					\begin{enumerate}
						\item It is obvious that $(L, \overline{R})$ is a Quillen adjoint pair since  $R$ preserves fibrations and trivial fibrations by hypothesis and $\overline{R}$ is aritywise determined by $R.$   Thus  $\overline{R}$ preserves fibrations and trivial fibrations.

						\item We now prove that the Quillen adjoint pair $(L, \overline{R})$  is a Quillen equivalence between the semi-model categories  $Op_\mathcal{C}$ and $Op_\mathcal{D}.$
						
						 \begin{enumerate}
							\item Let $P$ be a cofibrant operad on $\mathcal{C}$ and consider the fibrant resolution $L(P) \overset{\simeq}{\rightarrowtail} L(P)^{f}.$ 
							To prove that the composite 
							\begin{center}
								$P \longrightarrow \overline{R}L(P) \longrightarrow \overline{R}(L(P)^{f})$
							\end{center}
							is a weak equivalence, it is sufficient to prove that the following composite is a weak equivalence in $[FinSet, \mathcal{D}]$
						\begin{center}
							$\overline{\lambda}UP \longrightarrow UL(P) \longrightarrow U(L(P)^{f})$
						\end{center}	
						where $U: Op_{\mathcal{C}}\longrightarrow [FinSet, \mathcal{D}]$ is the forgetful functor.
				One can see that each of these two maps are weak equivalences for the following reasons: The first map of this composite is a weak equivalence according to Proposition \ref{Relation_betwen_left_adjoint_functors}. The second morphism is also a weak equivalence since in the semi-model structure defined on operads, the forgetful functor $U$ preserves weak equivalences. 
					
							\item  Let $Q$ be a fibrant operad on $\mathcal{D}$ and a cofibrant resolution $R(Q)^c \overset{\simeq}{\longrightarrow} R(Q).$
							
							To prove that 	\begin{center}
								$L(R(Q)^c) \longrightarrow L(R(Q)) \longrightarrow Q$
							\end{center}	is a weak equivalence in $Op_\mathcal{D},$							
							we consider the following commutative diagram
							
							\begin{center}
								$\xymatrix{ \overline{R}L(\overline{R}(Q)^c) \ar[r] & \overline{R}L\overline{R}(Q) \ar[r]& \overline{R}(Q)\\
									\overline{R}(Q)^c \ar[u]^{\eta} \ar[r]^{\simeq} & \overline{R}(Q)\ar[ru]_-{=} \ar[u]^{\eta} &
								}$
							\end{center}
							which is adjoint to the diagram in $[FinSet, \mathcal{D}]$

		\begin{center}
			$\xymatrix{ UL(\overline{R}(Q)^c) \ar[r] & UL\overline{R}(Q) \ar[r]& UQ\\
			\overline{\lambda}	(\overline{R}(Q)^c) \ar[u]^{\simeq} \ar[r]_{ }^{\beta} & 	\overline{\lambda}\overline{R}(Q)\ar[ru]_-{\alpha} \ar[u]^{ } &
			}$
		\end{center}	
		where the most right vertical weak equivalence is justified by Proposition \ref{Relation_betwen_left_adjoint_functors}.	
	Now, since the pair $(\overline{\lambda}, \overline{R})$	is a Quillen equivalence, one deduce that the composite $\alpha \beta$ is a weak equivalence. Therefore it follows from the 2-out-of-3 axiom for weak equivalences in $[FinSet, \mathcal{D}]$ that 
			\begin{center}
				$UL(R(Q)^c) \longrightarrow UL(R(Q)) \longrightarrow UQ$
			\end{center}
is a weak equivalence. 	The result then follows from the semi-model structure defined on operads.	

\end{enumerate}
\item We assume that the model structure on $Op_\mathcal{D}$ is a strict model structure and we want to prove that the pair $(L, \overline{R})$  is a strict Quillen equivalence.

	Let $P$ be a cofibrant operad on $\mathcal{C}=Ch_{\geq t}(\Bbbk), Q$ be a fibrant operad in $\mathcal{D},$ and consider a weak equivalence $\widetilde{f}:L(P) \overset{\simeq}{\longrightarrow} Q.$
  	
  	The adjoint of this map fits into the commutative diagram  	$\xymatrix{
  		P \ar[r]\ar[d]_{}  &\overline{R}(Q) \\
  		\overline{R}L(P) 	 \ar[ru]_{\overline{R}(\widetilde{f} )}& }$ 
  	
  	We aim to prove that the map $P \longrightarrow \overline{R}(Q)$ is a weak equivalence and  it will be enough in showing that the morphism $UP \longrightarrow U \overline{R}(Q)$ (where $U: Op_{\mathcal{C}}\longrightarrow [FinSet, \mathcal{C}]$ is the forgetful functor) is a weak equivalence. 
  	
  	The adjoint in $[FinSet, \mathcal{D}]$ of the diagram 
  	\begin{center}
  		$\xymatrix{
  			UP \ar[r]\ar[d]_{}  &U\overline{R}(Q) \\
  			U	\overline{R}L(C) 	 \ar[ru]_{\overline{R}(\widetilde{f} )}& }$ 
  	\end{center}
  	gives the diagram 
  	\begin{center}
  		$\xymatrix{
  			\overline{\lambda}U(P) \ar[r]^{(2)}\ar[d]_{(1)}  &UQ \\
  			UL(P) 	 \ar[ru]_{U\widetilde{f} }^{\simeq}& }$ 
  	\end{center}
  	We proved in Proposition \ref{Relation_betwen_left_adjoint_functors} that $(1)$ is a weak equivalence, therefore $(2)$ is a weak equivalence between the fibrant $UQ$ ( in $[FinSet, \mathcal{D}]$) and the cofibrant $UP$ ( in $[FinSet, \mathcal{C}]$) object. We now use the fact that the adjoint pair $(	\overline{\lambda}, 	\overline{R})$ is a Quillen equivalence (Proposition \ref{lift_adjonction_symmetric_sequence}) to deduce that $UP \longrightarrow U\overline{R}(Q)$ is a weak equivalence.
  	
  	Conversely, consider a quasi-isomorphism $g: P \overset{\simeq}{\longrightarrow} \overline{R}(Q).$ We want to prove that $\widetilde{g}: L(P)\overset{\simeq}{\longrightarrow} Q,$ and it will be sufficient to prove that  $\widetilde{g}: UL(P)\overset{\simeq}{\longrightarrow} UQ.$  Let us consider the commutative diagram
  	\begin{center}
  		$\xymatrix{\overline{\lambda}(UP)\ar[r]^{\alpha_0}_{\simeq}\ar[d]^{\overline{\lambda}(g)}&UL(P)\ar[d]^{\gamma_0} \ar[rd]^{\widetilde{g}}\\
  			\overline{\lambda}\overline{R}Q \ar[r]^{\alpha_1}& UL\overline{R}Q \ar[r]^{\gamma_1}& UQ 
  		}$
  	\end{center}
  	\begin{enumerate}
  		\item [-] $\alpha_0$ is a quasi-isomorphism according to Proposition \ref{Relation_betwen_left_adjoint_functors};
  		\item[-] $\gamma_1 \circ \alpha_1: \overline{\lambda} \overline{R} \longrightarrow UQ$ is the unit of the adjunction $(\overline{\lambda},  \overline{R}),$ therefore $\gamma_1  \alpha_1 \overline{\lambda}(g): \overline{\lambda}(UP) \longrightarrow UQ$ is the adjoint of $g.$ We then deduce that $ \gamma_1  \alpha_1 \overline{\lambda}(g)=\widetilde{g}\alpha_0 $ is a weak equivalence, since $(\overline{\lambda},  \overline{R})$ is a Quillen equivalence.
  		
  		Therefore from the $2$ out of $3$ property, one deduce that $\widetilde{g}$ is a weak equivalence.
  	\end{enumerate}

\end{enumerate}
					
\end{Proof of main theorem}

				\bibliographystyle{alpha}
				\bibliography{mybibliography}

\begin{thebibliography}{{Spi}01}

\bibitem[BM02]{Berger02axiomatichomotopy}
Clemens Berger and Ieke Moerdijk.
\newblock Axiomatic homotopy theory for operads.
\newblock {\em Comment. Math. Helv}, 78:020609--4, 2002.

\bibitem[BM06]{BergerMoerdjik}
Clemens Berger and Ieke Moerdijk.
\newblock The {B}oardman-{V}ogt resolution of operads in monoidal model
  categories.
\newblock {\em Topology}, 45(5):807--849, 2006.

\bibitem[Bor94]{Borc94}
Francis Borceux.
\newblock {\em Handbook of categorical algebra. 2}, volume~51 of {\em
  Encyclopedia of Mathematics and its Applications}.
\newblock Cambridge University Press, Cambridge, 1994.
\newblock Categories and structures.

\bibitem[Chi05]{Ching05}
Michael Ching.
\newblock Bar constructions for topological operads and the {G}oodwillie
  derivatives of the identity.
\newblock {\em Geom. Topol.}, 9:833--933, 2005.

\bibitem[Chi12]{Ching12}
Michael Ching.
\newblock A note on the composition product of symmetric sequences.
\newblock {\em J. Homotopy Relat. Struct.}, 7(2):237--254, 2012.

\bibitem[DS95]{Dwyer95homotopytheories}
W.~G. Dwyer and J.~Spalinski.
\newblock Homotopy theories and model categories, 1995.

\bibitem[FHT01]{FHT01}
Yves F\'elix, Stephen Halperin, and Jean-Claude Thomas.
\newblock {\em Rational homotopy theory}, volume 205 of {\em Graduate Texts in
  Mathematics}.
\newblock Springer-Verlag, New York, 2001.

\bibitem[Fre09a]{Fresse2008BenoitFM}
Benoit Fresse.
\newblock {\em Modules over operads and functors}, volume 1967 of {\em Lecture
  Notes in Mathematics}.
\newblock Springer-Verlag, Berlin, 2009.

\bibitem[Fre09b]{BF09}
Benoit Fresse.
\newblock Operadic cobar constructions, cylinder objects and homotopy morphisms
  of algebras over operads.
\newblock In {\em Alpine perspectives on algebraic topology}, volume 504 of
  {\em Contemp. Math.}, pages 125--188. Amer. Math. Soc., Providence, RI, 2009.

\bibitem[Fre17a]{fresse2017homotopy}
B.~Fresse.
\newblock {\em Homotopy of Operads and Grothendieck-Teichmuller Groups. {P}art
  1}.
\newblock Mathematical Surveys and Monographs. American Mathematical Society,
  2017.

\bibitem[Fre17b]{BF17}
Benoit Fresse.
\newblock {\em Homotopy of operads and {G}rothendieck-{T}eichm\"uller groups.
  {P}art 2}, volume 217 of {\em Mathematical Surveys and Monographs}.
\newblock American Mathematical Society, Providence, RI, 2017.
\newblock The applications of (rational) homotopy theory methods.

\bibitem[GG99]{GetzlerGoerss}
Ezra Getzler and Paul Goerss.
\newblock A model category structure for differential graded coalgebras.
\newblock {\em Preprint}, 1999.

\bibitem[Har10]{Harper}
John~E. Harper.
\newblock Homotopy theory of modules over operads and non-{$\Sigma$} operads in
  monoidal model categories.
\newblock {\em J. Pure Appl. Algebra}, 214(8):1407--1434, 2010.

\bibitem[Hir03]{Hir03}
Philip~S. Hirschhorn.
\newblock {\em Model categories and their localizations}, volume~99 of {\em
  Mathematical Surveys and Monographs}.
\newblock American Mathematical Society, Providence, RI, 2003.

\bibitem[Hov99]{Hovey98}
Mark Hovey.
\newblock {\em Model categories}, volume~63 of {\em Mathematical Surveys and
  Monographs}.
\newblock American Mathematical Society, Providence, RI, 1999.

\bibitem[HPS97]{Hovey_Palmieri}
Mark Hovey, John~H. Palmieri, and Neil~P. Strickland.
\newblock Axiomatic stable homotopy theory.
\newblock {\em Mem. Amer. Math. Soc.}, 128(610):x+114, 1997.

\bibitem[HRY17]{hackney2017shrinkability}
Philip {Hackney}, Marcy {Robertson}, and Donald {Yau}.
\newblock Shrinkability, relative left properness, and derived base change.
\newblock {\em The New York Journal of Mathematics}, 23:83--117, 2017.

\bibitem[Isa02]{DanielC2002}
Daniel~C. Isaksen.
\newblock Calculating limits and colimits in pro-categories.
\newblock {\em Fundamenta Mathematicae}, 175(2):175--194, 2002.

\bibitem[ML63]{maclane}
Saunders Mac~Lane.
\newblock {\em Homology}.
\newblock Die Grundlehren der mathematischen Wissenschaften, Bd. 114. Academic
  Press, Inc., Publishers, New York; Springer-Verlag,
  Berlin-G\"ottingen-Heidelberg, 1963.

\bibitem[Mur14]{Muro}
Fernando Muro.
\newblock Homotopy theory of non-symmetric operads, {II}: {C}hange of base
  category and left properness.
\newblock {\em Algebr. Geom. Topol.}, 14(1):229--281, 2014.

\bibitem[PS18]{Pavlov_Jakob}
Dmitri Pavlov and Jakob Scholbach.
\newblock Admissibility and rectification of colored symmetric operads.
\newblock {\em Journal of Topology}, 11(3):559--601, 2018.

\bibitem[Qui67]{Quill67}
Daniel~G. Quillen.
\newblock {\em Homotopical algebra}.
\newblock Lecture Notes in Mathematics, No. 43. Springer-Verlag, Berlin-New
  York, 1967.

\bibitem[Rez96]{Rezk}
Charles~W Rezk.
\newblock Spaces of algebra structures and cohomology of operads.
\newblock {\em Ph.D. thesis, MIT}, 1996.

\bibitem[{Spi}01]{Spitzweck}
Markus {Spitzweck}.
\newblock Operads, algebras and modules in general model categories.
\newblock {\em arXiv preprint arXiv:math/0101102}, 2001.

\bibitem[SS03]{SS03}
Stefan Schwede and Brooke Shipley.
\newblock Equivalences of monoidal model categories.
\newblock {\em Algebr. Geom. Topol.}, 3:287--334, 2003.

\bibitem[Swe69]{Sweedler1969hopf}
M.E. Sweedler.
\newblock {\em Hopf algebras}.
\newblock Mathematics lecture note series. W. A. Benjamin, 1969.

\bibitem[WY16]{white2016homotopical}
David {White} and Donald {Yau}.
\newblock Homotopical adjoint lifting theorem.
\newblock {\em arXiv preprint arXiv:1606.01803}, 2016.

\end{thebibliography}

			\end{document}